\documentclass[regno,a4paper]{amsart}

\usepackage{amsmath, amssymb, amscd, amsthm,epsfig,epic,eepic}
\usepackage[justify]{ragged2e}

\usepackage[T1]{fontenc}

\usepackage{framed}
\usepackage{multicol}
\usepackage[normalem]{ulem}
\usepackage{booktabs}
\usepackage{multirow}
\usepackage{bbding}
\useunder{\uline}{\ul}{}
\usepackage{mathrsfs}
\usepackage{mathdots}
\usepackage[shortlabels]{enumitem}
\usepackage[new]{old-arrows}

\makeatletter
     \let\oldfootnote\footnote
     \def\footnote{\@ifstar\footnote@star\footnote@nostar}
     \def\footnote@star#1{{\let\thefootnote\relax\footnotetext{#1}}}
     \def\footnote@nostar{\oldfootnote}
     \makeatother

\newcommand{\N}{\mathbf{N}}
\newcommand{\Z}{\mathbf{Z}}

\makeatletter
\newcommand*{\defeq}{\mathrel{\rlap{%
                      \raisebox{0.3ex}{$\cdot$}}%
                      \raisebox{-0.3ex}{$\cdot$}}%
                      =}
\makeatother

\newcommand{\proj}{\mathbf{P}}
\newcommand{\Gm}{\mathbf{G}_m}

\newcommand{\Ga}{\mathbf{G}_a}

\newcommand{\dd}{\mathop{}\!\text{d}}

\DeclareMathOperator{\Proj}{Proj}

\DeclareMathOperator{\Aut}{Aut}

\DeclareMathOperator{\sm}{sm}
\DeclareMathOperator{\Stab}{Stab}
\DeclareMathOperator{\Supp}{Supp}

\DeclareMathOperator{\height}{ht}

\DeclareMathOperator{\Lie}{Lie}

\DeclareMathOperator{\SL}{SL}

\DeclareMathOperator{\Sp}{Sp}

\DeclareMathOperator{\Spin}{Spin}

\DeclareMathOperator{\rank}{rank}

\usepackage{xcolor, color}
\definecolor{links}{HTML}{A93226}
\definecolor{bluetto}{HTML}{1ABC9C}
\definecolor{cite}{HTML}{21618C}
\usepackage{hyperref, cleveref}
\hypersetup{
    colorlinks=true,
    linkcolor=cite,
    citecolor=links,
    filecolor=links,
    urlcolor=links,
    linktocpage=true
}

\usepackage[inner=3cm,outer=3cm,top=2.75cm,bottom=2.75cm]{geometry}

\usepackage[all]{xy}
\usepackage{tikz-cd}

\theoremstyle{plain}

\newtheorem{theorem}{Theorem}[section]
\newtheorem{lemma}[theorem]{Lemma}
\newtheorem{proposition}[theorem]{Proposition}
\newtheorem{corollary}[theorem]{Corollary}
\newtheorem{theorem*}{Theorem}
\newtheorem{proposition*}[theorem*]{Proposition}
\newtheorem{lemma*}[theorem*]{Lemma}
\newtheorem*{proposition**}{Proposition}

\theoremstyle{definition}
\newtheorem{definition}[theorem]{Definition}
\newtheorem{remark}[theorem]{Remark}

\newtheorem{example}[theorem]{Example}
\numberwithin{equation}{section}

\begin{document}

\title{PARABOLIC SUBGROUPS IN CHARACTERISTICS TWO AND THREE}

\author{Matilde Maccan}
\email{matilde.maccan@univ-rennes.fr}

\footnote*{Keywords: parabolic subgroup scheme, positive characteristic, homogeneous space}
\footnote*{2020 Mathematics Subject Classification:
Primary: 14M17, Secondary: 14L15, 17B20}
\date{\today}

\maketitle

\begin{abstract}
This text brings to an end the classification of non-reduced parabolic subgroups in positive characteristic, especially two and three: they are all obtained as intersections of parabolics having maximal reduced part. We prove this result and deduce a few geometric consequences on rational projective homogeneous varieties.
\end{abstract}

\tableofcontents

\section*{Introduction}
We work in the setting of affine group schemes of finite type
over an algebraically closed field $k$, for which a general reference is \cite{Milne}. We are interested in studying rational projective varieties, homogeneous under their automorphism group: they can be realised as quotients of semisimple groups by parabolic subgroup schemes.\\
In characteristic zero, fixing a semisimple group, a Borel subgroup and a maximal torus $G \supset B \supset T$, there is a bijection between parabolic subgroups containing $B$ and subsets of the set of simple roots of $G$: a parabolic subgroup is determined by the basis of the root system of a Levi subgroup. On the other hand, over a field of positive characteristic
  parabolic subgroups \emph{can be nonreduced}, which yields a richer behaviour and different geometric properties of homogeneous spaces; see for example \cite[Section 4]{HL93}, \cite{Lauritzen} and \cite[Theorem 3.1]{Totaro}.\\\\
  If the characteristic is at least $5$, or if the Dynkin diagram of $G$ is simply laced,
  Wenzel \cite{Wenzel}, Haboush and Lauritzen \cite{HL93} show that all parabolic subgroups of $G$ can be obtained from reduced ones by fattening with Frobenius kernels and intersecting. More precisely, let us denote as $G^m$ the kernel of the $m$-th iterated Frobenius morphism and $P^\alpha$ the maximal reduced parabolic subgroup generated by $B$ and all root subgroups $U_{-\beta}$ with $\beta$ ranging over the simple roots distinct from $\alpha$. Then all parabolic subgroups are of the form 
  \begin{align}
  \label{stand}
  G^{m_1}P^{\beta_1} \cap \ldots \cap G^{m_r} P^{\beta_r},
  \end{align}
  where $\beta_1,\ldots,\beta_r$ are simple roots of $G$. We call a parabolic of such form \emph{of standard type}.\\
  The proof of \cite{Wenzel} relies heavily on the structure constants 
  relative to a Chevalley basis of the Lie algebra of a simply connected semisimple group. By construction, such constants are independent of the characteristic and are integers with absolute value strictly less than five. The hypotheses that either $p\geq 5$ or that the Dynkin diagram has no multiple edges both guarantee that these constants do not vanish over $k$.\\\\
In \cite{Maccan}, we classify parabolic subgroups with maximal reduced part, focusing on non-simply laced groups in characteristic two or three: except for a group of type $G_2$ in characteristic two, these parabolics are all obtained by fattening $P^\alpha$ with the kernel of some class of isogenies (those with no central factor, a notion which slightly generalises the Frobenius morphism and which only arises in small characteristics). For the $G_2$ case, two exotic families of parabolic subgroups enter into the picture, coming from the existence of two maximal $p$-Lie subalgebras. We describe them explicitly, study the two underlying varieties, and show that they complete the classification in Picard rank one.\\\\
In this text, we are able to go further and get the full list of parabolic subgroups, in all types and characteristics. The main result is the following.
\begin{theorem*}
    \label{main_intro}
    Let $P$ be a parabolic subgroup of a semisimple group $G$ and let $\beta_1,\ldots, \beta_r$ be the simple roots of $G$ such that
    \[
    P_{\text{red}} = \bigcap_{i=1}^r P^{\beta_i}.
    \]
    Then
    \[
    P = \bigcap_{i=1}^r Q^i,
    \]where $Q^i$ is the smallest subgroup containing $P$ and $P^{\beta_i}$.\\
    In particular, $P$ is intersection of parabolic subgroups with maximal reduced part.
\end{theorem*}

Let us mention that, even if $P$ is of the form (\ref{stand}), some parabolic subgroups not of standard type can appear in its expression as intersection of the $Q^i$s; this is illustrated in \Cref{Qforstandard}.\\\\
This text is organised as follows. 
In Section $1$, we collect groundwork material, which can be found in more detail in \cite{Maccan}; starting with some notation and results of \cite{Wenzel}, we then make the reduction to treating the case of a simple, simply connected group. Next, we move on to the fundamental notion of \emph{very special isogeny}, introduced and studied by \cite{BorelTits} and \cite{CGP15}, which allows us to state a result on the factorisation of isogenies.
Moreover, we recall the notion of Schubert divisors and Schubert curves on $G/P$, as well as the construction of a family of contraction morphisms, indexed by the set of simple roots $\beta_i$ which do not belong to the Levi subgroup of $P_\text{red}$. These are defined geometrically, and have as targets the varieties
\[
G/Q^i
\]
whose stabilisers are exactly the subgroups appearing in the statement of \Cref{main_intro}.\\
We conclude with the following key Lemma.
\begin{lemma*}
 Let $P$ be a parabolic subgroup of $G$ as above and consider some simple root $\beta_i$. Then 
 \[U_{-\beta_i} \cap P = U_{-\beta_i} \cap Q^i.\]
\end{lemma*}
Section $2$ is dedicated to the proof of \Cref{main_intro}. First, we make a few remarks on the associated numerical functions of $P$ and of the $Q^i$s; following \cite{Wenzel}, we define such a function as
\[
\varphi \colon \Phi \longrightarrow \N \cup \{\infty \}, \quad \gamma \longmapsto \height(P\cap U_{-\gamma}),
\]
where the \emph{height} of a subgroup is the smallest non-negative integer $m$ such that the $m$-th iterated Frobenius morphism kills it. Then, we proceed by a case-by-case analysis according to the Dynkin diagram of $G$ and the characteristic of the base field, because the proof heavily depends on the structure constants of $\Lie G$.\\\\
We gather in Section $3$ a few geometric consequences of \Cref{main_intro}; first, we compute the subgroups $Q^i$ for a parabolic subgroup of standard type, then we study whether there exists a smooth contraction morphism on a homogeneous projective variety: it turns out that this is always the case, except for the exotic cases in type $G_2$.\\\\
In characteristic zero, every $G/P$ is Fano, and there are only a finite number with fixed dimension. Both of these statements prove false in positive characteristics; nevertheless, as the final result presented in this text, we prove the following finiteness property:
\begin{theorem*}
    Let $n \geq 1$ be a fixed integer.\\
    There are a finite number of isomorphism classes of projective homogeneous varieties of dimension $n$ which are Fano.
\end{theorem*}

\textbf{Acknowledgments.} I would like to express my gratitude to my PhD advisors, Michel Brion and Matthieu Romagny, for all the time they are investing in guiding my work and for their constant support.


\section{Preliminaries}

\subsection{Notation and conventions} 
In this text, $k$ is an algebraically closed field of prime characteristic $p >0$. Unless otherwise stated, every object is going to be defined over the base field $k$. When $V$ is a finite-dimensional $k$-vector space, we denote $\proj(V)$ to be the projective space of lines in $V$.\\\\
Let $G \supset B \supset T$ be respectively a semisimple, simply connected algebraic group, a Borel subgroup and a maximal torus contained in it. We are interested in rational projective homogeneous $G$-varieties, which one can realise as quotients of the form $G/P$ where $P$ is a parabolic subgroup of $G$, \emph{not necessarily reduced}: by subgroup we always mean a closed subgroup scheme; by parabolic subgroup we mean one containing a Borel subgroup. Since Borel subgroups are all conjugate, we can restrict to subgroups containing $B$.\\ 
Let us call $\Phi$ the root system of $G$ relative to $T$, $\Phi^+$ the subset of positive roots associated to the Borel subgroup $B$ and $\Delta$ the corresponding basis of simple roots. For any root $\gamma$, let $\Supp(\gamma)$ be its support, defined as the set of simple roots which have a nonzero coefficient in the expression of $\gamma$ (as linear combination of simple roots with integer coefficients of the same sign). Moreover, we denote as $U_\gamma$ the root subgroup associated to $\gamma$, as $u_\gamma \colon \Ga \rightarrow U_\gamma$ its root homomorphism and as $\mathfrak{g}_\gamma$ the corresponding one-dimensional Lie subalgebra of $\Lie G$. Let us fix a Chevalley basis of $\Lie G$:
\[
\{ X_\gamma, H_\alpha \}_{\gamma \in \Phi, \alpha\in\Delta}.\]
    In particular, 
    \[
    \mathfrak{g}_\gamma = \Lie U_\gamma = k X_\gamma \quad \text{and} \quad X_\gamma = \dd u_\gamma (1).
    \]
We denote as $P_I$ the reduced parabolic subgroup having $I \subset \Delta$ as set of simple roots of a Levi subgroup. 
    Moving on to less standard notation, $G^m$ denotes the kernel of the $m$-th iterated Frobenius morphism 
\[
F^m_G \colon G \rightarrow G^{(m)},
\]
while the maximal reduced parabolic subgroup not containing $U_{-\alpha}$, for $\alpha$ a simple root, is
\[
P^{\alpha} \defeq P_{\Delta \backslash \{ \alpha \} }.
\]
For a non-reduced parabolic subgroup $P$, let 
\begin{align}
\label{infinitesimal}
U_P^- \defeq P \cap R_u^-(P_{\text{red}})
\end{align}
be its intersection with the unipotent radical of the opposite reduced parabolic of $P_{\text{red}}$. The subgroup $U_P^-$ is unipotent, infinitesimal and satisfies
\begin{align}
\label{product}
U_P^- = \prod_{\gamma \in \Phi^+ \backslash \Phi_I} (U_P^- \cap U_{-\gamma}) \quad \text{and} \quad P = U_P^- \times P_{\text{red}},
\end{align}
where both identities are isomorphisms of schemes given by the multiplication of $G$: this is proven in \cite[Proposition 4 and Theorem 10]{Wenzel}. 
Concerning root systems, we follow conventions from \cite{Bourbaki}. For two subgroups $H$ and $K$ of $G$, we denote as 
\[
\langle H,K\rangle
\]
the smallest subgroup of $G$ containing both of them.\\
The aim of this text is to prove the following classification result.

\begin{theorem}
\label{mainnnn}
    Let $P$ be a parabolic subgroup of a semisimple group $G$ over an algebraically closed field of any characteristic, with reduced part $P_I$. Then the inclusion
    \begin{align}
    \label{inclusion0}
    P \subseteq \bigcap_{\alpha \in \Delta \backslash I} \langle P,P^\alpha \rangle
    \end{align}
    is an equality; in particular, $P$ is intersection of parabolic subgroups with maximal reduced part.
\end{theorem}

\begin{lemma}
\label{Gsimple}
    Any parabolic subgroup $P$ of a semisimple simply connected group 
\[
G = G_1 \times \cdots \times G_n
\]
is a product of the parabolic subgroups $P_i \defeq P\cap G_i$ of the simple factors $G_i$.
\end{lemma} 

In particular, from now on we can (and will) assume the group $G$ to be simple.

\begin{proof}
Indeed, by the second isomorphism in (\ref{product}), we can see $P$ as a product of its unipotent infinitesimal part and its reduced part. Moreover, by the first isomorphism in (\ref{product}), we have that 
\[
U_P^- = U_{P_1}^- \times \cdots \times U_{P_n}^-,
\]
while it is a classical fact that
\[
P_{\text{red}} = (P_1)_{\text{red}}  \times \cdots \times (P_n)_{\text{red}}.
\]
Putting these equalities together allows to conclude that $P$ is the product of the $P_i$s.
\end{proof}

\begin{definition}
\label{standardtype}
    A parabolic subgroup is said to be \emph{of standard type} if it is of the form
    \[
    P = \bigcap_{\alpha \in \Delta \backslash I} G^{m_\alpha} P^\alpha
    \]
    for some non-negative integers $m_\alpha$.
\end{definition}

The (partial) classification, which has been established in \cite{Wenzel} and \cite{HL93}, can be reformulated as follows: when the characteristic is at least equal to $5$ or when the group $G$ is simply laced, all parabolic subgroups are of standard type. Let us notice that, for a parabolic subgroup of standard type, the integers $m_\alpha$ are uniquely determined as being the height of the intersection 
\[
U_{-\alpha} \cap P.
\]

\subsection{Isogenies}
We introduce here the notions which allow us to slightly generalise \Cref{standardtype} in order to extend the classification to small characteristics.

\begin{definition}
If the Dynkin diagram of $G$ has an edge of multiplicity $p$ (namely $p=2$ for types $B_n$, $C_n$ and $F_4$, $p=3$ for type $G_2$) we say that $G$ satisfies the \emph{edge hypothesis}. Under this assumption, there is a unique minimal noncentral subgroup $N_G \subset G$ of height one with Lie algebra
\[
\Lie N_G = \langle \Lie \gamma^\vee (\Gm) \colon \gamma \text{ short} \rangle  \bigoplus_{\gamma \text{ short}} \mathfrak{g}_\gamma.
\]
The corresponding quotient 
\[
\pi_G \colon G \longrightarrow G/N_G =: \overline{G}
\]
is called the \emph{very special isogeny} of $G$.
\end{definition}
In particular, the group $\overline{G}$ is still simple and simply connected, and its root system is dual to the one of $G$; moreover, the composition 
\[
\pi_{\overline{G}} \circ \pi_G
\]
is equal to the Frobenius morphism of $G$. See \cite{BorelTits} for the original work on the subject, or \cite{CGP15} for more details. When the group $G$ is implicit from the context, we make an abuse of notation and simply write $F$ for $F_G$, $N$ for $N_G$, $\pi$ for $\pi_G$ and so on.\\
The notion of very special isogeny allows to state the following factorisation result, (see \cite[Proposition 2]{Maccan}).

\begin{proposition}
\label{factorisation}
Let us consider an isogeny
\[\xi \colon G \rightarrow G^\prime.\]
Then there exists a unique factorisation of $\xi$:
\begin{itemize}
    \item either as the composition of the very special isogeny, followed by an iterated Frobenius morphism and then a central isogeny (which can only occur when the edge hypothesis is satisfied by $G$);
    \item or as the composition of an iterated Frobenius morphism with a central isogeny.
\end{itemize}
\end{proposition}

\begin{remark}
\label{minimale}
In particular, this allows us to speak of isogenies \emph{with no central factor}, which are the ones we are interested in when classifying parabolic subgroups. Kernels of such isogenies are totally ordered by inclusion, as follows:
\[
1 \subsetneq N \subsetneq G^1 \subsetneq N^1 \subsetneq \ldots \subsetneq G^m \subsetneq N^m \subsetneq G^{m+1} \subsetneq \ldots,
\]
where $N^m_G$ (which we denote simply by $N^m$ when $G$ is implicit) is the kernel of the composition of a very special isogeny and an $m$-th iterated Frobenius morphism.
\end{remark}

\begin{definition}
    A parabolic subgroup is said to be \emph{of quasi-standard type} if it is of the form
    \[
    P = \bigcap_{\alpha \in \Delta \backslash I} (\ker \xi_\alpha ) P^\alpha
    \]
    for some isogenies $\xi_\alpha$ with no central factor.
\end{definition}

Let us notice that by \Cref{factorisation}, the subgroup $\ker \xi_\alpha$ is necessarily of the form $G^mP^\alpha$ or $N^m_G P^\alpha$. Moreover, the latter are enough to classify all parabolic subgroups having maximal reduced part, in all types except for $G_2$, in view of the following result (see \cite[Proposition 4]{Maccan}):

\begin{proposition}
\label{kerphi}
Let $P$ be a parabolic subgroup of $G$ such that its reduced subgroup is equal to $P^\alpha$ for some simple root $\alpha$.\\
Then there exists a unique isogeny $\xi$ with no central factor such that
\[
P = (\ker \xi) P^\alpha,
\]
unless $G$ is of type $G_2$ in characteristic $2$ and $\alpha$ is the short simple root.
\end{proposition}




\begin{remark}
\label{minimality}
With respect to the order given in \Cref{minimale}, $\xi$ is \emph{minimal}: in other words, for any isogeny $\zeta$ such that $\ker \zeta$ is strictly contained in $\ker \xi$, the subgroup $P$ is not contained in $(\ker \zeta)P^\alpha$. This is a crucial point in the proof of \Cref{mainnnn}.
\end{remark}


\subsection{Contractions}
\label{sec:contractions}
Consider a homogeneous variety
\[
X= G/P, \quad \text{with } P_{\text{red}} = P_I
\] 
and denote as $o$ its base point. We can define the Schubert divisors of $X$, associated to simple roots, as
\begin{align}
\label{Dalpha}
D_\alpha\defeq \overline{B^- s_\alpha o}, \quad \alpha \in \Delta \backslash I.
\end{align}
In \cite[Lemma 3.19]{Maccan}, we build a finite family of morphisms
\begin{align}
\label{f_alpha}
f_\alpha \colon X  \longrightarrow G/Q^\alpha \defeq \Proj \bigoplus_{m\geq 0} H^0(X, \mathcal{O}_X(mD_\alpha)), \quad \alpha \in \Delta \backslash I.
\end{align}
We construct $f_\alpha$ as the the unique contraction on $G/P$ such that the Schubert curves 
\[
C_\beta \defeq \overline{U_{-\beta} o}, \quad \beta \in \Delta \backslash I,
\]
which are smooth, are all contracted to a point except for $C_\alpha$. Then we show that $Q^\alpha$ is the subgroup generated by $P$ and $P^\alpha$.

\begin{lemma}
    In all types except for $G_2$ in characteristic $2$, there is a unique isogeny $\xi_\alpha$ with no central factor, such that 
    \[
Q^\alpha = (\ker \xi_\alpha) P^\alpha.
\]
\end{lemma}

\begin{proof}
    By construction, $Q^\alpha$ has maximal reduced part equal to $P^\alpha$, hence we can apply \Cref{kerphi} and \Cref{minimality} and we are done.
\end{proof}



\subsection{Height on simple root subgroups}
Let us start by proving a Lemma which is repeatedly used in the proof of \Cref{mainnnn}. This result tells us that the inclusion (\ref{inclusion0}) is \textit{not so far} from being an equality: more precisely, when intersecting with the root subgroup associated to the opposite of a simple root, one gets the same height on both sides.

\begin{lemma}
\label{lem:m_alpha}
 Let $P$ be any parabolic subgroup of $G$ and $\alpha$ a simple root. Then 
 \[U_{-\alpha} \cap P = U_{-\alpha} \cap Q^\alpha.\]
\end{lemma}

\begin{proof}
The natural $G$-equivariant morphism
\[
g_\alpha \colon X = G/P \longrightarrow \proj(H^0(X, \mathcal{O}_X(D_\alpha))^\vee)
\]
is well-defined since $D_\alpha$ is globally generated. The canonical section of $H^0(X,\mathcal{O}_X(D_\alpha))$, corresponding to a hyperplane $H_\alpha$, gives the equality
\begin{align}
    \label{galpha}
    g_\alpha^\ast \mathcal{O}(H_\alpha) = \mathcal{O}_X(D_\alpha).
\end{align}
On the other hand, the inclusion of the $k$-subalgebra generated by elements of degree one into the direct sum $\oplus_{m\geq 0} H^0(X, \mathcal{O}_X(mD_\alpha))$ gives a finite morphism $h_\alpha$, making the following diagram commute.
\[
\begin{tikzcd}
    X = G/P \arrow[r, "g_\alpha"] \arrow[dd, "f_\alpha"] & \proj(H^0(X, \mathcal{O}_X(D_\alpha))^\vee)\\
    &\\
    \Proj \bigoplus_{m\geq 0} H^0(X, \mathcal{O}_X(mD_\alpha)) = G/Q^\alpha \arrow[uur, "h_\alpha"] &
\end{tikzcd}
\]
Since $f_\alpha$ is a contraction and $h_\alpha$ is finite, the above diagram is the Stein factorisation of the morphism $g_\alpha$. Let us denote as $E_\alpha$ and $S_\alpha$ respectively the Schubert divisor and the Schubert curve in $G/Q^\alpha$. Then set-theoretically the pre-image of $E_\alpha$ is $D_\alpha$, while the image of $C_\alpha$ is $S_\alpha$. This means that there are some positive integers $m_\alpha$ and $n_\alpha$ such that $f_\alpha^\ast E_\alpha = m_\alpha D_\alpha$ and $(f_\alpha)_\ast C_\alpha = n_\alpha S_\alpha$. The equality (\ref{galpha}) yields 
\[
D_\alpha = g^\ast_\alpha H_\alpha = f^\ast_\alpha h_\alpha^\ast H_\alpha,\]
hence $m_\alpha$ must be equal to $1$. Moreover, 
\[
1 = D_\alpha \cdot C_\alpha = f^\ast_\alpha E_\alpha \cdot C_\alpha = E_\alpha \cdot (f_\alpha)_\ast C_\alpha = n_\alpha E_\alpha \cdot S_\alpha = n_\alpha.
\]
This means that $f_\alpha$ restricts to an isomorphism from $C_\alpha$ to $S_\alpha$: considering the restriction to the respective affine open cells, we get
\[
U_{-\alpha} / ( U_{-\alpha} \cap P) = U_{-\alpha} / (U_{-\alpha} \cap Q^\alpha)
\]
as wanted.
\end{proof}

\section{Classification}

This section is dedicated to the proof of \Cref{mainnnn}. By \Cref{Gsimple}, we may and will assume the group $G$ to be simple.

\subsection{Ingredients for the proof}
We will use the classification of \cite[Proposition 3 and 4]{Maccan}, as well as \Cref{lem:m_alpha} and the factorisation of isogenies. Before moving on to the proof, which is a case-by-case argument, let us fix some common notation and state a few remarks of which we repetitively make use below.\\
We associate to $P$ the numerical function
\[
\varphi \colon \Phi \longrightarrow \N \cup \{\infty \}
\]
defined as follows: for a root $\gamma$, the integer $\varphi(\gamma)$ is the height of the intersection 
\[
U_{-\gamma} \cap P\]
when $\gamma$ is positive and not in the Levi subgroup of $P_I$; and we extend it to $\varphi(\gamma) = \infty$ otherwise. Analogously, $\varphi_i$ denotes the associated function to $Q^i \defeq Q^{\alpha_i}$, where we keep the notation of \cite{Bourbaki} concerning the standard bases of root systems.

\begin{remark}
\label{rosa_nero_verde}
$(a)$ : In particular, we can reformulate \Cref{lem:m_alpha} as 
\[
\varphi (\alpha_i) = \varphi_i(\alpha_i) = \height((\ker \xi_i) \cap U_{-\alpha_i}).
\]
$(b)$ : 
    Let us recall that two parabolic subgroups $P$ and $P^\prime$, with respective associated functions $\varphi$ and $\varphi^\prime$ and with same reduced part $P_I$, 
    coincide if and only if 
    \[
    P\cap U_{-\gamma} = P^\prime \cap U_{-\gamma} \quad \text{for all } \gamma \in \Phi^+ \backslash \Phi_I,
    \]
    which is equivalent to saying that $\varphi$ and $\varphi^\prime$ are equal. This is proven in \cite[Proposition 8]{Wenzel}.\\
    $(c)$ : The functions $\varphi_i$ always coincide on roots of the same length. More precisely, if the isogeny $\xi_i$ is an iterated Frobenius, then $\varphi_i$ is constant on all positive roots not in $\Phi_I$. On the other hand, if $\xi_i$ is the composition of an $m$-th iterated Frobenius morphism and a very special isogeny, then
    \[
    m+1 = \varphi_i(\gamma) = \varphi_i(\delta)+1
    \]
    for all short roots $\gamma$ and all long roots $\delta$ not belonging to $\Phi_I$. The two integers above are invariant under the action of the Weyl group, which has one orbit if $G$ is simply laced and two orbits (of long and short roots respectively) otherwise. 
\end{remark}

Let us recall the following result on structure constants (see \cite[Chapter VII, $25.2$]{Humphreys}).

\begin{lemma}
Let $\gamma \neq \pm \delta$ be roots and $r$ a natural number such that 
\[
\gamma-r\delta, \ldots, \gamma,\gamma+\delta
\]
are roots but $\gamma-(r+1)\delta$ is not. 
Then the corresponding vectors of the Chevalley basis satisfy
\[
[X_\gamma,X_\delta]= \pm (r+1) X_{\gamma+\delta}.
\]
\end{lemma}

The integer $r$ only depends on the roots $\gamma$ and $\delta$, and it is uniquely defined because we are working with a fixed Chevalley basis. We denote it as 
\[\mathcal{N}(\gamma,\delta)\]
in what follows. This allows us to formulate the following slight adaptation of an argument in \cite{Wenzel}. 

\begin{lemma}
\label{enne}
    Let $\gamma$ and $\delta$ be roots such that $\gamma+\delta$ is a root but $\gamma-\delta$ is not. Then
    \[
\varphi(\gamma+\delta) \geq \min \{ \varphi(\gamma),\varphi(\delta) \}.
\]
\end{lemma}

\begin{proof}
    Let $m \defeq \varphi(\gamma+\delta)$ be the height of 
    \[
    U_{-\gamma-\delta}\cap P.\]
    Let us consider $a,b \in \Ga$ such that $u_{-\gamma}(a)$ and $u_{-\delta}(b)$ are in $P$. Then the commutator
    \[
    (u_{-\gamma}(a),u_{-\delta}(b)) = \prod 
    u_{-i\gamma-j\delta} (c_{ij}a^ib^j),
    \]
     where the product ranges over the finite set of couples of positive integers $(i,j)$ such that $i\gamma+j\delta$ is a root, has a factor $u_{-\gamma-\delta}(\pm ab)$, because 
     \[
     c_{11} = \mathcal{N}(\gamma,\delta) = \pm 1.
     \]
     By \cite[Proposition $8$]{Wenzel}, $u_{-\gamma-\delta}(ab)$ belongs to $U_{-\gamma-\delta} \cap P$, hence $(ab)^{p^m}$ vanishes. In particular either $a^{p^m}$ or $b^{p^m}$ is zero, which means that the minimum between the height of $U_{-\gamma} \cap P$ and $U_{-\delta} \cap P$ is less than or equal to $m$, as wanted. 
\end{proof}


\subsection{Type $B_n$ and $C_n$}
In this section we consider a simply connected group $G$ of type $B_n$ or $C_n$, over an algebraically closed field of characteristic $p=2$. Let
\begin{itemize}
    \item $\overline{P}$ be the pull-back of $P$ via the very special isogeny 
    \[
    \overline{\pi} \defeq \pi_{\overline{G}} \colon \overline{G} \longrightarrow G;
    \]
    \item $\alpha_i \leftrightarrow \overline{\alpha}_i$ the bijection on simple roots induced by $\overline{\pi}$, which we recall exchanges long and short roots;
    \item $\psi$ and $\psi_i$ the respective associated functions to $\overline{P}$ and to
    \[
    \ker (\xi_i \circ \overline{\pi})P^{\overline{\alpha}_i} = \langle Q,P^{\overline{\alpha}_i} \rangle.
    \]
\end{itemize}
The situation is summarized in the following diagram.

\begin{center}
    \begin{tikzcd}
     \overline{P} \arrow[r, hookrightarrow] \arrow[dd] & \bigcap_{\alpha_i \in \Delta \backslash I} \ker (\xi_i \circ \overline{\pi}) P^{\overline{\alpha}_i} \arrow[r, hookrightarrow]  & \overline{G} \arrow[dd,"\overline{\pi}"]\\
     &&\\
        P \arrow[r, hookrightarrow]  & \bigcap_{\alpha_i \in \Delta \backslash I} (\ker \xi_i)P^{\alpha_i} \arrow[r, hookrightarrow]  & G
    \end{tikzcd}
\end{center}

\begin{lemma}
    \label{BimpliesC}
    If (\ref{inclusion0}) is an equality for all parabolic subgroups of a group of type $B_n$, then the same holds in type $C_n$.
\end{lemma}

\begin{proof}
    We make use of the above diagram: let $G$ be of type $C_n$, then by assumption (\ref{inclusion0}) holds for the parabolic subgroup $\overline{P}$, which means that
    \[
    \psi(\overline{\gamma}) = \min_i\{ \psi_i(\overline{\gamma})\}
    \]
    for all positive roots $\gamma$ of $G$, where $\gamma \leftrightarrow \overline{\gamma}$ is the bijection induced by the very special isogeny. This implies
    \begin{align*}
        \varphi(\gamma) & = \psi(\overline{\gamma}) = \min_i \{ \psi_i(\overline{\gamma}) \} = \min_i \{ \varphi_i(\gamma) \} & \text{if $\gamma$ is short,}\\
        \varphi(\gamma) & = \psi(\overline{\gamma}) -1 = \min_i \{ \psi_i(\overline{\gamma}) -1 \} = \min_i \{ \varphi_i(\gamma) \} & \text{if $\gamma$ is long,}
    \end{align*}
    so we are done.
\end{proof}
In order to prove \Cref{mainnnn} for a group of type $B_n$, we proceed by induction on $n$, of which the first step is the case $n=2$ below.

\begin{lemma}
\label{initialisation}
    Let 
    \[P \subset (\ker \xi_1)P^{\alpha_1} \cap (\ker \xi_2)P^{\alpha_2}\subset \Spin_5
    \]
    be a parabolic subgroup in type $B_2$, with reduced part the Borel subgroup.\\
    Then (\ref{inclusion0}) is an equality.
\end{lemma}

\begin{proof}
Below is a picture of the respective root systems (which are isomorphic) in order to visualise the very special isogeny and make the proof easier to read: a basis in type $B_2$ is given by a long root $\alpha_1$ and a short root $\alpha_2$, while in type $C_2$ it is given by a short root $\overline{\alpha}_1$ and a long root $\overline{\alpha}_2$.
\begin{center}
\begin{tikzpicture}
    \foreach\ang in {90,180,270,360}{
     \draw[->,cite!80!black,thick] (0,0) -- (\ang:2.4cm);
    }
    \foreach\ang in {45,135,225,315}{
     \draw[->,cite!80!black,thick] (0,0) -- (\ang:3.3cm);
    }
    \node[anchor= west,scale=0.9] at (2.5,0) {$\alpha_1+\alpha_2$};
    \node[anchor= south,scale=0.9] at (2.5,2.5) {$\alpha_1+2\alpha_2$};
    \node[anchor= south,scale=0.9] at (0,2.5)  {\colorbox{cite!30!white}{$\alpha_2$}};
    \node[anchor= west,scale=0.9] at (2.5,-2.5) {\colorbox{cite!30!white}{$\alpha_1$}};
    \node[anchor= north,scale=0.9] at (-2.5,1.7)  {\textbf{TYPE $\mathbf{B_2}$}};
  \end{tikzpicture}
  \hfill
  \begin{tikzpicture}
    \foreach\ang in {90,180,270,360}{
     \draw[->,cite!80!black,thick] (0,0) -- (\ang:2.6cm);
    }
    \foreach\ang in {45,135,225,315}{
     \draw[->,cite!80!black,thick] (0,0) -- (\ang:2cm);
    }
    \node[anchor= west,scale=0.9] at (2.6,0) {$2 \overline{\alpha}_1 +\overline{\alpha}_2$};
    \node[anchor= west,scale=0.9] at (1.6,-1.6) {\colorbox{cite!30!white}{$\overline{\alpha}_1$}};
    \node[anchor= south,scale=0.9] at (1.6,1.6) {$\overline{\alpha}_1 +\overline{\alpha}_2$};
     \node[anchor= south,scale=0.9] at (0,2.6) {\colorbox{cite!30!white}{$\overline{\alpha}_2$}};
     \node[anchor= north,scale=0.9] at (-2,2.1)  {\textbf{TYPE $\mathbf{C_2}$}};
  \end{tikzpicture}
\end{center}
   By \Cref{rosa_nero_verde} (b), it is enough to prove the following:
    \begin{align}
    \label{uno}
        \varphi (\alpha_1+\alpha_2) \geq \min \{ \varphi_1(\alpha_1+\alpha_2), \varphi_2(\alpha_1+\alpha_2) \};\\
    \label{due}
        \varphi (\alpha_1+2\alpha_2) \geq  \min \{ \varphi_1(\alpha_1+2\alpha_2), \varphi_2(\alpha_1+2\alpha_2) \},
    \end{align}
   because the opposite inequalities are already implied by the inclusion (\ref{inclusion0}).\\
   Let us consider the two integers
    \[
    r_1 \defeq \varphi_1(\alpha_1) 
    \quad \text{and} \quad r_2 \defeq \varphi_2(\alpha_2).
    \]
    Given that $\alpha_1$ is a long root, we have either
    \begin{align}
        \label{xi_one}
        \xi_1 = F^{r_1} \quad \text{or} \quad \xi_1 = F^{r_1} \circ \pi,
    \end{align}
    because by \Cref{factorisation} these are the only two isogenies with no central factor whose kernel has height $r_1$ on long root subgroups. Analogously, $\alpha_2$ being a short root, we have either 
     \begin{align}
        \label{xi_two}
        \xi_2 = F^{r_2} \quad \text{or} \quad \xi_2 = F^{r_2-1} \circ \pi.
    \end{align}
    for the same reason.\\
   \textbf{Step 1:} Let us start by considering the root $\alpha_1+\alpha_2$. We have that $\mathcal{N}(\alpha_1,\alpha_2) = \pm 1$, so 
    \begin{align*}
        \varphi(\alpha_1+\alpha_2) & \geq \min \{ \varphi(\alpha_1), \varphi(\alpha_2) \} & \text{by \Cref{enne}},\\
        & = \min \{ \varphi_1(\alpha_1) ,\varphi_2(\alpha_2) \}& \text{by \Cref{rosa_nero_verde} (a)}.
    \end{align*}
    This translates into the inequality
    \begin{align}
        \label{cerchio}
        \varphi (\alpha_1+\alpha_2) \geq \min \{r_1,r_2\}.
    \end{align}
    Since $\alpha_2$ and $\alpha_1+\alpha_2$ are both short, we have 
    \[
    \varphi_2(\alpha_1+\alpha_2) = r_2
    \]
    by \Cref{rosa_nero_verde}(c). Moreover, if $\xi_1$ is an $r_1$-th Frobenius iterated, then
    \[
    \varphi_1(\alpha_1+\alpha_2) = r_1
    \] so that by (\ref{cerchio}) we are done. So let us assume that 
    \begin{align}
         \label{somethin}
    \xi_1 = F^{r_1} \circ \pi,
    \end{align}
    which in particular means 
    \[
    \varphi_1(\alpha_1+\alpha_2)= r_1+1.
    \]
    What is left to prove in this case is that 
    \[
    \varphi (\alpha_1+\alpha_2) \geq \min \{r_1+1,r_2\}.
    \]
    Now, if $r_2 \leq r_1$, then (\ref{cerchio}) becomes
    \[
    \varphi(\alpha_1+\alpha_2) \geq \min \{r_1,r_2\}  = r_2 = \min \{r_1+1,r_2\}
    \]
    and we are done. So let us assume that $r_1<r_2$: again by (\ref{cerchio}), it is enough to get a contradiction with the assumption
    \[
    \varphi(\alpha_1+\alpha_2) = r_1.
    \]
    Let us assume this last equality to be true, and remark that by (\ref{xi_two}), since $\alpha_1+2\alpha_2$ is a long root we have 
    \[
    \varphi_2(\alpha_1+2\alpha_2 ) \leq r_2
    \]
    in both cases. Next, the inclusion (\ref{inclusion0}) gives
    \[
    \varphi(\alpha_1+2\alpha_2) \leq \min \{ \varphi_1 (\alpha_1+2\alpha_2), \varphi_2(\alpha_1+2\alpha_2)\} \leq \min \{r_1,r_2\} = r_1.
    \]
    In particular, all positive roots $\gamma$ whose support contains $\alpha_1$ satisfy $\varphi(\gamma) \leq r_1$. In other words, the subgroup $P$ is contained in $G^{r_1} P^{\alpha_1}$. However, this implies that 
    \[
    \langle P, P^{\alpha_1}\rangle = (\ker \xi_1)P^{\alpha_1} \subset G^{r_1} P^{\alpha_1},
    \]
    which contradicts (\ref{somethin}).\\\\
    \textbf{Step 2:} Let us move on to the root $\alpha_1+2\alpha_2$: consider the pull-back
    \[
    \overline{P} \defeq \overline{\pi}^{-1}(P) \subseteq \ker (\xi_1 \circ \overline{\pi} ) P^{\overline{\alpha}_1} \cap \ker (\xi_2 \circ \overline{\pi}) P^{\overline{\alpha}_2}, 
    \]
    which is a parabolic subgroup of $\Sp_4$.\\
    The very special isogeny $\overline{\pi}$ sends $\overline{\alpha}_1 +\overline{\alpha}_2$ to $\alpha_1+2\alpha_2$: in particular, pulling back via $\overline{\pi}$ gives
    \begin{align}
    \label{pallino0}
    \psi(\overline{\alpha}_1+\overline{\alpha}_2) = \varphi(\alpha_1+2\alpha_2) +1 
    \end{align}
    and the analogous equalities hold for $\psi_1$ and $\psi_2$. Thus proving (\ref{due}) is equivalent to showing that
    \begin{align}
    \label{pallino}
        \psi(\overline{\alpha}_1 +\overline{\alpha}_2) \geq \min \{ \psi_1(\overline{\alpha}_1+\overline{\alpha}_2 ) ,\psi_2(\overline{\alpha}_1+\overline{\alpha}_2) \}.
    \end{align}
    Now, the structure constant $\mathcal{N}(\overline{\alpha}_1,\overline{\alpha}_2)$ is equal to $\pm 1$, so we have
    \begin{align*}
         \psi (\overline{\alpha}_1 +\overline{\alpha}_2)& \geq \min \{ \psi(\overline{\alpha}_1),\psi(\overline{\alpha}_2)\} &  \text{by \Cref{enne}}\\
        & = \min \{ \psi_1(\overline{\alpha}_1),\psi_2(\overline{\alpha}_2)\} & \text{by \Cref{rosa_nero_verde} (a)}.
    \end{align*}
    This translates into the inequality
    \begin{align}
        \label{aiut}
        \psi(\overline{\alpha}_1+\overline{\alpha}_2) \geq \min \{ r_1+1,r_2\}.
    \end{align}
    Next, $\overline{\alpha}_1 +\overline{\alpha}_2$ and $\overline{\alpha}_1$ are both short, which implies
    \[ 
    \psi_1(\overline{\alpha}_1+\overline{\alpha}_2) = \psi_1(\overline{\alpha}_1)  = r_1+1
    \]
    by \Cref{rosa_nero_verde}(c). On the other hand, $\overline{\alpha}_2$ is long, hence (again by \Cref{factorisation})
    \[
    \psi_2(\overline{\alpha}_1 +\overline{\alpha}_2) = \begin{cases}
        \psi_2(\overline{\alpha}_2) = r_2 \quad \quad \text{ if } \xi_2 \circ \overline{\pi} = F^{r_2};\\
        \psi_2(\overline{\alpha}_2)+1 = r_2 +1 \quad \text{ if } \xi_2 = F^{r_2}.
    \end{cases}
    \]
    In the first case, (\ref{pallino}) automatically holds, so let us assume we are in the second one, namely that the isogeny $\xi_2$ is an $r_2$-th iterated Frobenius, which in particular means that 
    \[
    \psi_2(\overline{\alpha}_1+\overline{\alpha}_2 ) = r_2+1.
    \]
    What is left to prove in this case is that
    \[
    \psi(\overline{\alpha}_1 +\overline{\alpha}_2 )\geq \min \{ r_1+1,r_2+1\}.
    \]
    Now, if $r_2>r_1$, then (\ref{aiut}) becomes
    \[
    \psi(\overline{\alpha}_1+\overline{\alpha}_2) \geq \min \{ r_1+1,r_2\} = r_1+1 = \min \{r_1+1,r_2+1\}
    \]
    and we are done. So let us assume that $r_2\leq r_1$. Again by (\ref{aiut}), it is enough to get a contradiction with the assumption
    \[
    \psi (\overline{\alpha}_1+\overline{\alpha}_2) =r_2.
    \]
    If this last assumption holds, then pushing forward to $P$ and using (\ref{pallino0}) gives
    \[
    \varphi(\alpha_1+2\alpha_2) = \psi (\overline{\alpha}_1+\overline{\alpha}_2) -1= r_2-1.
    \]
    Putting this together with (\ref{uno}), which has been proved in Step 1 above, yields
    \[
    \varphi(\alpha_1+\alpha_2) = \min \{ \varphi_1(\alpha_1+\alpha_2), \varphi_2(\alpha_1+\alpha_2)\} \leq r_2.
    \]
    In particular, the subgroup $P$ is contained in $\ker(F^{r_2-1}\circ \overline{\pi}) P^{\alpha_2}$,
    thus implying 
    \[
    \langle P, P^{\alpha_2}\rangle =  (\ker \xi_2)P^{\alpha_1} \subset \ker(F^{r_2-1}\circ \overline{\pi}) P^{\alpha_2},
    \]
    which contradicts the assumption that $\xi_2 = F^{r_2}$.
\end{proof}

\begin{proposition}
    Let $P$ be as in the diagram above in type $B_n$ or $C_n$.\\
    Then (\ref{inclusion0}) is an equality.
\end{proposition}

\begin{proof}
    \Cref{BimpliesC} and \Cref{initialisation} imply that the statement holds for $n= 2$.\\
    Let us assume it to be true in types $B_n$ and $C_n$, with respective basis
    \[
    \alpha_1, \, \ldots, \, \alpha_n \quad \text{and} \quad \overline{\alpha}_1, \, \ldots, \, \overline{\alpha}_n,\]
    and consider the groups of type $B_{n+1}$ and $C_{n+1}$, with respective basis 
    \[
    \alpha_0,\, \alpha_1, \, \ldots, \, \alpha_n \quad \text{and} \quad \overline{\alpha}_0,\, \overline{\alpha}_1,\, \ldots, \, \overline{\alpha}_n.
    \]
    First, by \Cref{BimpliesC} we can reduce to the case of a group of type $B_{n+1}$. Moreover, by \Cref{rosa_nero_verde}(b), 
    it suffices to show that 
    \begin{align}
    \label{tre}
    \varphi(\gamma) = \min \{ \varphi_i (\gamma) \colon \alpha_i \in \Supp(\gamma)\} \quad \text{for all } \gamma \in \Phi^+\backslash \Phi_I^+.  
    \end{align}
    \textbf{Step 1:} Let us assume that the simple root $\alpha_0$ is not in the support of $\gamma$ and let $L^{\alpha_0}$ be the Levi subgroup associated to $\alpha_0$ (namely, with basis all simple roots except for the first one): then $L^{\alpha_0}$ is of type $B_n$ and the inclusion
    \[
    U_{-\gamma} \subset L^{\alpha_0}
    \]
    is satisfied. We can thus apply the induction hypothesis and conclude that (\ref{tre}) holds for $\gamma$.\\\\
    \textbf{Step 2:} Let us consider a root of the form
    \[
    \gamma = \varepsilon_0-\varepsilon_i, \quad 1 \leq i\leq n.
    \]
    Denoting as $L^{\alpha_n}$ the Levi subgroup associated to the last simple root $\alpha_n$, we see that $L^{\alpha_n}$ is of type $A_{n-1}$ hence simply laced; moreover, the inclusion
    \[
    U_{-\gamma} \subset L^{\alpha_n}
    \]
    is satisfied. In particular, the intersection $P\cap L^{\alpha_n}$ is a parabolic subgroup of $L^{\alpha_n}$, hence of standard type, from which we deduce the equality (\ref{tre}) for $\gamma$.\\\\
   \textbf{Step 3:} Next, let us consider the root \[
   \gamma = \varepsilon_0.\]
    The structure constant $\mathcal{N}(\varepsilon_0-\varepsilon_1,\varepsilon_1)$ is equal to $\pm 1$, so we have
     \begin{align*}
        \varphi(\varepsilon_0) & \geq \min \{ \varphi(\varepsilon_0-\varepsilon_1), \varphi(\varepsilon_1) \} & \text{by \Cref{enne}}\\
        & = \min \{ \varphi_0(\alpha_0) ,\varphi(\varepsilon_1) \}& \text{by \Cref{rosa_nero_verde} (a)}\\
        & = \min \{\varphi_0(\alpha_0), \varphi_i(\varepsilon_1), \, i >0\} & \text{by Step 1}\\
        & = \min \{ \varphi_0(\alpha_0), \varphi_i(\varepsilon_0), \, i>0\} & \text{by \Cref{rosa_nero_verde}(c)}.
    \end{align*}
    If $\alpha_0$ belongs to $I$ then we are done because the minimum does not involve $\varphi_0(\alpha_0)$. Hence we can assume $\alpha_0$ to be in $\Delta \backslash I$ and set
    \[
    r_0 \defeq \varphi_0(\alpha_0) \quad \text{and} \quad r \defeq \min \{\varphi_i (\varepsilon_0), \, i>0\}.
    \]
    Since $\alpha_0$ is a long root, by \Cref{factorisation} the isogeny $\xi_0$ is either equal to an $r_0$-th iterated Frobenius or to its composition with a very special isogeny. In the first case we directly have (\ref{tre}) and we are done, so let us assume that
    \begin{align}
    \label{somethin0}
    \xi_0 = F^{r_0} \circ \overline{\pi}.
    \end{align}
    In particular, $\varphi_0(\varepsilon_0)$ is equal to $m+1$. Thus, the inequalities above become
    \[
    \varphi(\varepsilon_0) \geq \min \{r_0,r\}.
    \]
    If $r$ is less than or equal to $r_0$ we are again done, so let us assume $r_0<r$. We want to exclude the possibility of a strict inequality i.e.
    \[
    r_0 = \varphi(\varepsilon_0).
    \]
    However, since $\varepsilon_0$ is the only short root containing $\alpha_0$ in its support, this last assumption implies that $P$ is contained in $G^{r_0} P^{\alpha_0}$. In particular, 
    \[
    \langle P, P^{\alpha_0} \rangle = (\ker \xi_0)P^{\alpha_0} \subset G^m P^{\alpha_0},
    \]
    which contradicts (\ref{somethin0}). 
    This allows to conclude that (\ref{tre}) is always true for the root $\varepsilon_0$.\\\\
   \textbf{Step 4:} We have left to prove (\ref{tre}) for the roots
   \[
   \gamma = \varepsilon_0+\varepsilon_i, \quad 1 \leq i\leq n.\]
   In this case, the support is equal to the whole of $\Delta$. If $i \neq 1$, we can use the fact that the structure constant $\mathcal{N}(\varepsilon_0-\varepsilon_1,\varepsilon_1+\varepsilon_i)$ is equal to $\pm 1$  to obtain
   \begin{align*}
        \varphi(\varepsilon_0+\varepsilon_i) & \geq \min \{ \varphi(\varepsilon_0-\varepsilon_1), \varphi(\varepsilon_1+\varepsilon_i) \} & \text{by \Cref{enne}}\\
        & = \min \{ \varphi_0(\alpha_0) ,\varphi(\varepsilon_1+\varepsilon_i) \}& \text{by \Cref{rosa_nero_verde} (a)}\\
        & = \min \{\varphi_0(\alpha_0), \varphi_l(\varepsilon_1+\varepsilon_i), \, l >0\} & \text{because (\ref{tre}) holds for } \varepsilon_1+\varepsilon_i\\
        & = \min_j \{ \varphi_j(\varepsilon_0+\varepsilon_i)\} & \text{by \Cref{rosa_nero_verde}(c)}.
    \end{align*}
    On the other hand, if $i=1$, we can proceed analogously using the fact that the structure constant  $\mathcal{N}(\varepsilon_0-\varepsilon_n,\varepsilon_1+\varepsilon_n)$ is equal to $\pm 1$, and we get 
    \begin{align*}
        \varphi(\varepsilon_0+\varepsilon_1) & \geq \min \{ \varphi(\varepsilon_0-\varepsilon_n), \varphi(\varepsilon_1+\varepsilon_n) \} & \text{by \Cref{enne}}\\
        & = \min \{\varphi_j(\varepsilon_0-\varepsilon_n), \, j<n, \, \varphi(\varepsilon_1+\varepsilon_n) \} & \text{because (\ref{tre}) holds for } \varepsilon_0-\varepsilon_n\\
        & = \min \{\varphi_j(\varepsilon_0-\varepsilon_n), \, j<n, \, \varphi_k(\varepsilon_1+\varepsilon_n), \, k >0\} & \text{because (\ref{tre}) holds for } \varepsilon_1+\varepsilon_n\\
        & = \min_j \{ \varphi_j(\varepsilon_0+\varepsilon_1)\} & \text{by \Cref{rosa_nero_verde}(c)}.
    \end{align*}
    This concludes the proof.
\end{proof}
    
\begin{lemma}
    Let us assume that we are working over a field of characteristic $p=3$. Then any parabolic subgroup of a simple group of type $B_n$ or $C_n$ is standard. In particular, \Cref{mainnnn} also holds in this case.
\end{lemma}
We follow the reasoning in \cite[Theorem 10]{Wenzel}, and we use the fact that all structure constants appearing in the proof have absolute value strictly smaller than $3$. 

\begin{proof}
    Keeping the above notation, the inclusion (\ref{inclusion0}) becomes
    \[
    P \subseteq \bigcap_{\alpha \in \Delta \backslash I} G^{m_\alpha}P^\alpha,
    \]
    where, by \Cref{rosa_nero_verde}(a) we have
    \[
    m_\alpha = \varphi(\alpha) = \varphi_\alpha(\alpha).
    \]
    It is thus enough to prove the inequality
    \[
    \varphi(\gamma) \geq \min \{m_\alpha \colon \alpha \in \Supp(\gamma)\},
    \]
    by induction on the height of $\gamma$. If $\gamma$ is simple then the statement is true. If not, let $\alpha \in \Supp(\gamma)$ such that $\gamma-\alpha$ is still a root. Then the structure constant
    \[
    \mathcal{N}(\gamma-\alpha,\alpha)
    \]
    is equal to $\pm 1$ or to $\pm 2$; in particular it never vanishes over the base field. This implies, by the same reasoning as the one in the proof of \Cref{enne}, that
    \[
    \varphi(\gamma) \geq \min \{ \varphi(\gamma-\alpha),\varphi(\alpha)\},
    \]
    and we are done by the induction hypothesis.
\end{proof}



\subsection{Type $F_4$}
Consider a group $G$ of type $F_4$ over an algebraically closed field of characteristic $p>0$ (in particular $p=2$ or $p=3$ are the interesting cases for us).\\
Let us adapt all previous notation:
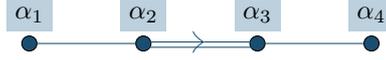
\begin{center}
\begin{tikzpicture}[inner sep=2pt,outer sep=0pt]
\node [circle,fill=cite!80!black,radius=2pt,draw,label=\colorbox{cite!30!white}{$\alpha_1$}] (A) at (0,0) {};
\node [circle,fill=cite!80!black,radius=2pt,draw,label=\colorbox{cite!30!white}{$\alpha_2$}] (B) at (1.5,0) {};
\node (C) at (2.25,0) {};
\node [circle,fill=cite!80!black,radius=2pt,draw,label=\colorbox{cite!30!white}{$\alpha_3$}] (D) at (3,0) {};
\node [circle,fill=cite!80!black,radius=2pt,draw,label=\colorbox{cite!30!white}{$\alpha_4$}] (E) at (4.5,0) {};
\draw[thin,color=cite!80!black] (A) -- (B);
\draw[thin,color=cite!80!black] (D) -- (E);
\draw[-{Classical TikZ Rightarrow[length=1.5mm,]},thin,color=cite!80!black] (1.5,0.02) -- (2.3,0.02);
\draw[thin,color=cite!80!black] (2.3,0.02) -- (3,0.02);
\draw[thin,color=cite!80!black] (2.25,-0.05) -- (3,-0.05);
\draw[thin,color=cite!80!black] (1.5,-0.05)--(2.25,-0.05);
\end{tikzpicture}
\end{center}
is the Dynkin diagram, and we denote as
\[
Q^i = \langle P, P^{\alpha_i} \rangle = (\ker \xi_i)P^{\alpha_i}, \quad \alpha_i \in \Delta \backslash I,
\]
the family of parabolic subgroups with maximal reduced part associated to $P$. Moreover, we call $\varphi$ and $\varphi_i$ the associated functions to $P$ and to $\langle P , P^{\alpha_i}\rangle$ respectively.

\begin{proposition}
    The inclusion (\ref{inclusion0}) is an equality in type $F_4$.
\end{proposition}

\begin{proof}
    By \Cref{rosa_nero_verde}(b), it is enough to show that
    \begin{align}
        \label{cinque}
        \varphi(\gamma) \geq \min_i \{ \varphi_i(\gamma) \}, \quad \text{for all } \gamma \in \Phi^+\backslash \Phi^+_I.
    \end{align}
\textbf{Step 1:} Let us assume that the support of $\gamma$ is not equal to the whole of $\Delta$. Then $U_{-\gamma}$ is contained in a Levi subgroup of the form
\[
L \defeq P^{\alpha_i} \cap (P^{\alpha_i})^-
\]
for some $i$. In particular, $L$ has root system of type $C_3$ (if $\alpha_1$ is not in $\Supp(\gamma)$), of type $B_3$ (if $\alpha_4$ is not in  $\Supp(\gamma)$) or of type $A_1 \times A_2$ (if one among $\alpha_2$ and $\alpha_3$ is not in $\Supp(\gamma)$). In any of those situations, the inequality (\ref{cinque}) holds for $\gamma$, because 
\[
U_{-\gamma} \cap P = U_{-\gamma} \cap L\]
have same height, and because any parabolic subgroup in type $C_3$, $B_3$ and $A_1 \times A_2$ is intersection of parabolic subgroups with maximal reduced part.\\\\
\textbf{Step 2:} Let us consider a long root $\gamma$ whose support is equal to $\Delta$. By \Cref{D4} proven below, the subgroup $H$ generated by long root subgroups is of type $D_4$. Let us consider the following basis for the root system of $H$, as in (\ref{beta_i}):
\begin{align*}
& \beta_1 \defeq \alpha_2+2\alpha_3+\alpha_4 = \varepsilon_1-\varepsilon_2, \quad \quad \beta_2 \defeq \alpha_1 = \varepsilon_2-\varepsilon_3,\\
& \beta_3 \defeq \alpha_2 = \varepsilon_3-\varepsilon_4, \qquad \qquad \qquad \quad \beta_4 \defeq \alpha_2+2\alpha_3 = \varepsilon_3+\varepsilon_4.
\end{align*}
Since $\gamma$ is a long root satisfying
\[
U_{-\gamma} \cap P = U_{-\gamma} \cap H,
\]
we can use the fact that the group of type $D_4$ is simply laced to get
\begin{align}
\label{diciotto}
\varphi(\gamma) = \min_i \{ \varphi(\beta_i)\}.
\end{align}
Next, let us consider again the $\beta_i$s as being roots of $G$; notice that their support is not equal to the whole of $\Delta$. This allows us to apply Step 1 to $\beta_i$, to obtain
\[
\varphi(\beta_i) = \min_j \{\varphi_j(\beta_i) \colon \alpha_j \in \Supp (\beta_i)\}, \quad 1\leq i \leq 4.
\]
In conclusion, (\ref{diciotto}) becomes
\[
\varphi(\gamma ) = \min_{i,j} \{ \varphi_j(\beta_i) \colon \alpha_j \in \Supp (\beta_i) \} = \min_j \{ \varphi_j(\gamma) \} 
\]
where the last equality is due to \Cref{rosa_nero_verde}(c); in particular, (\ref{cinque}) holds for $\gamma$.\\\\
\textbf{Step 3:} Finally, we are led to consider short roots whose support is equal to the whole of $\Delta$. There are five of them, namely:
\begin{align*}
    \delta_1 & = \alpha_1+\alpha_2+\alpha_3+\alpha_4 = \delta_2-\alpha_3,\\
    \delta_2 & = \alpha_1+\alpha_2+2\alpha_3+\alpha_4=\delta_3-\alpha_2,\\
    \delta_3 & = \alpha_1+2\alpha_2+2\alpha_3+\alpha_4 = \delta_4-\alpha_3,\\
    \delta_4 & =\alpha_2+2\alpha_2+3\alpha_3+\alpha_4 = \delta_5-\alpha_4,\\
    \delta_5 & = \alpha_1+2\alpha_2+3\alpha_3+2\alpha_4 = \delta_1+ (\alpha_2+2\alpha_3+\alpha_4).
\end{align*}
Let us recall that we set the function $\varphi$ to be constant and equal to infinity on negative roots. Moreover, we have 
\begin{align*}
     & \mathcal{N}(\delta_2,-\alpha_3) = \pm 1, & \text{because $\delta_2+\alpha_3$ is not a root;}\\
    & \mathcal{N}(\delta_3,-\alpha_2) = \pm 1, & \text{because $\delta_3+\alpha_2$ is not a root;}\\
    & \mathcal{N}(\delta_4,-\alpha_3) = \pm 1, & \text{because $\delta_4+\alpha_3$ is not a root;}\\
    & \mathcal{N}(\delta_5,-\alpha_4) = \pm 1, & \text{because $\delta_5+\alpha_4$ is not a root;}\\
    & \mathcal{N}(\delta_1, \nu) = \pm 1, \text{where } \nu = \alpha_2+2\alpha_3+\alpha_4, & \text{because $\delta_1+\nu$ is not a root.}
\end{align*}
The structure constants just above imply, by \Cref{enne}, that
\begin{align}
    \label{sei}
    \varphi(\delta_1) \geq \varphi(\delta_2) \geq \varphi(\delta_3) \geq \varphi(\delta_4) \geq \varphi(\delta_5) \geq \min \{ \varphi(\delta_1),\varphi(\nu)\}.
\end{align}
We can now apply Step 1 to the root $\nu$, because its support is not the whole of $\Delta$, as well as \Cref{rosa_nero_verde}(c), to obtain
\[\varphi_i(\nu) = \varphi_i(\delta_1)\]
for all $i$. In particular,
\begin{align*}
 \min \{ \varphi(\delta_1),\varphi(\nu)\} & = \min \{ \varphi(\delta_1), \varphi_2(\nu),\varphi_3(\nu), \varphi_4(\nu)\}\\ 
 & = \min \{ \varphi(\delta_1), \varphi_2(\delta_1),\varphi_3(\delta_1), \varphi_4(\delta_1) \} = \varphi(\delta_1).
\end{align*}
Together with (\ref{sei}), we can deduce that
\[
\varphi(\delta_1) = \varphi(\delta_2) = \varphi(\delta_3) =\varphi(\delta_4) = \varphi(\delta_5),
\]
so that
 it is enough to show
\begin{align}
\label{sette}
\varphi(\delta_1) \geq \min_i \{ \varphi_i(\delta_1) \} =: m.
\end{align}
Let us prove (\ref{sette}): first, let us notice that we have
\begin{align*}
    \varphi(\delta_1) & \geq  \min \{ \varphi(\alpha_1+\alpha_2+\alpha_3), \varphi(\alpha_4) \}& \text{because } \mathcal{N}(\alpha_1+\alpha_2+\alpha_3,\alpha_4) = \pm 1 ; \\
     & \geq \min \{ \varphi(\alpha_1+\alpha_2),\varphi(\alpha_3),\varphi(\alpha_4)\} & \text{because } \mathcal{N}(\alpha_1+\alpha_2,\alpha_3)= \pm 1 ;\\
      & \geq \min \{ \varphi(\alpha_1),\varphi(\alpha_2),\varphi(\alpha_3),\varphi(\alpha_4)\} & \text{because } \mathcal{N}(\alpha_1,\alpha_2) = \pm 1 ;\\
       & = \min_i \{ \varphi_i(\alpha_i) \} & \text{by \Cref{rosa_nero_verde}(a)};\\
        & = \min_i\{ \varphi_1(\alpha_1),\varphi_2(\alpha_2),\varphi_3(\delta_1), \varphi_4(\delta_1)\} & \text{by \Cref{rosa_nero_verde}(c)}.
\end{align*}
If the minimum just above is realised by $\varphi_3(\delta_1)$ or $\varphi_4(\delta_1)$, then (\ref{sette}) holds and we are done. If the minimum is realised by $\varphi_1(\alpha_1)$ and $\xi_1$ is an $m$-th iterated Frobenius, then we have
\[
m = \varphi_1(\alpha_1) = \varphi_1(\delta_1)
\]
and we are also done; analogously for $\varphi_2(\alpha_2)$.\\
We are left with the following cases, for which (\ref{sette}) becomes a strict inequality:
\begin{align*}
    & (a) \, \colon m = \varphi_2(\delta_1) > \varphi_2(\alpha_2) = m-1, \quad \text{i.e. } 
    \xi_2 = F^{m-1} \circ \overline{\pi}, \text{ or }\\
    & (b) \, \colon m = \varphi_1(\delta_1) > \varphi_1(\alpha_1) = m-1, \quad \text{i.e. } 
    \xi_1 = F^{m-1} \circ \overline{\pi}.
\end{align*}
Let us assume we are in one of these two situations. We want to get a contradiction with the assumption (respectively):
\begin{align*}
    & \text{in } (a), \, m-1= \varphi(\delta_1) < \varphi_2(\delta_1) =m;\\
    & \text{in }(b), \, m-1= \varphi(\delta_1) < \varphi_1(\delta_1) =m.
\end{align*}
We now claim that in both situations $(a)$ and $(b)$ all positive short roots $\gamma$ containing $\alpha_2$ in their support satisfy 
\begin{align}
\label{dieci}
\varphi(\gamma) \leq m-1.
\end{align}
Notice that in the root system of type $F_4$, all short roots containing $\alpha_1$ in their support also contain $\alpha_2$, so that this reasoning works for both $(a)$ and $(b)$.\\
The inequality (\ref{dieci}) is true for $\delta_1\ldots \delta_5$ by (\ref{sei}), hence we can consider $\gamma$ such that its support is not equal to $\Delta$. These roots are
\begin{align*}
    \gamma_1 & = \alpha_1+\alpha_2+\alpha_3, \, \text{satisfying } \varphi(\gamma_1) \leq \varphi(\delta_1) & \text{because } \mathcal{N}(\alpha_1+\alpha_2+\alpha_3, \alpha_4) = \pm 1;\\
     \gamma_2 & = \alpha_2+\alpha_3, \, \text{satisfying } \varphi(\gamma_2) \leq \varphi(\gamma_1) & \text{because } \mathcal{N}(\alpha_2+\alpha_3, \alpha_1) = \pm 1;\\
    \gamma_3 & = \alpha_2+\alpha_3+\alpha_4, \, \text{satisfying } \varphi(\gamma_3) \leq \varphi(\delta_1) & \text{because } \mathcal{N}(\alpha_2+\alpha_3+\alpha_4, \alpha_1) = \pm 1;\\
     \gamma_4 & = \alpha_1+2\alpha_2+\alpha_3, \, \text{satisfying } \varphi(\gamma_4) \leq \varphi(\gamma_3) & \text{because } \mathcal{N}(\alpha_2+2\alpha_3+\alpha_4, -\alpha_3) = \pm 1.
\end{align*}
In particular, this yields, using \Cref{enne} at each step, that
\begin{align*}
    & P \subseteq G^{m-1} P^{\alpha_2} \quad \text{in case } (a),\\
    & P \subseteq G^{m-1} P^{\alpha_1} \quad \text{in case } (b).
\end{align*}
This would imply that the kernel of $\xi_2$ (resp. of $\xi_1$) is contained in $G^{m-1}$, which contradicts the assumption $(a)$ (resp. the assumption $(b)$).
Hence (\ref{sette}) holds and we are done.
\end{proof}

\begin{lemma}
\label{D4}
    Let $G$ be of type $F_4$. Then the root subgroups associated to long roots of $G$ generate a reductive subgroup of type $D_4$.
\end{lemma}

\begin{proof}
    We proceed in two consecutive steps: the first one realises an embedding of $K = \Spin_9$ into $G$, while the second one an embedding of $H = \Spin_8$ into $K$, such that the root system of $H$ consists exactly of all the long roots of $G$.\\
    \textbf{Step 1:} Keeping the notation for root systems of \cite{Bourbaki}, we have that
    \begin{align}
        \label{B4}
        \pm \varepsilon_i, \quad \pm \varepsilon_i \pm \varepsilon_j, \, i \neq j
    \end{align}
    form a root subsystem of $\Phi$ of type $B_4$, with basis
    \[
     \nu_1 \defeq \varepsilon_1-\varepsilon_2 = \alpha_2+2\alpha_3+2\alpha_4, \quad \nu_2 \defeq \varepsilon_2-\varepsilon_3 = \alpha_1,\quad
        \nu_3 \defeq \varepsilon_3-\varepsilon_4 = \alpha_2,  \quad \nu_4 \defeq \varepsilon_4 = \alpha_3.
    \]
    The subgroup generated by these roots in the character lattice of the maximal torus $T$ of $G$ is
    \[
    R \defeq \langle \nu_1,\nu_2,\nu_3,\nu_4\rangle = \langle \alpha_1,\alpha_2,\alpha_3,2\alpha_4\rangle \subset X(T) 
    \]
    which has index $2$. The quotient map
    \[
    X(T) \longrightarrow \Z /2\Z 
    \]
    corresponds to an injection $M \subset T$, where $M$ is a copy of $\mu_2$.\\
    Let $K$ be the connected component of the identity of the centralizer of $M$ in $G$: by \Cref{reductive_centralizer} below, $K$ is smooth and reductive. Its Lie algebra satisfies
    \[
    \Lie K = \Lie C_G(M) = (\Lie G)^M = \Lie T \oplus \{ \mathfrak{g}_\gamma \, \colon \, tXt^{-1} = X \text{ for all } X \in \mathfrak{g}_\gamma, \, t \in M \}.
    \]
    Let us consider some $t \in M$ and $X \in \mathfrak{g}_\gamma$; by definition of root subspaces, we have 
    \[tXt^{-1} = \gamma(t) X.
    \]
    Moreover, by construction of $M$ we have $\gamma(t) \in \mu_2$ and $\gamma(M) = 1$ if and only if the coefficient of $\alpha_4$ (in the unique expression of $\gamma$ as linear combination of simple roots with integer coefficients of the same sign) is even. This means exactly that $\mathfrak{g}_\gamma$ is contained in $\Lie K$ if and only if $\gamma$ is one of the roots in (\ref{B4}). Finally, we can conclude that $K$ is simply connected of type $B_4$, with the desired set of roots and with maximal torus 
    \[
    T^\prime \defeq (T \cap K )^0_{\text{red}}.
    \]
    \textbf{Step 2:} The second step follows by the exact same reasoning, by considering 
    \[
    \pm \varepsilon_i \pm \varepsilon_j, \, i \neq j
    \]
    as a root subsystem of type $D_4$, with basis
    \begin{align}
    \label{beta_i}
    \beta_1 \defeq \nu_1, \quad \beta_2 \defeq \nu_2, \quad \beta_3 \defeq \nu_3, \quad \beta_4 \defeq \varepsilon_3+\varepsilon_4 = \nu_3+2\nu_4.
    \end{align}
    These roots generate
    \[
    R^\prime \defeq \langle \beta_1,\beta_2,\beta_3,\beta_4 \rangle = \langle \nu_1,\nu_2,\nu_3, 2\nu_4 \rangle \subset R
    \]
    as a subgroup of index $2$. The corresponding quotient map
    \[
    X(T^\prime ) = R \longrightarrow \Z/2\Z
    \]
    corresponds to an injection $M^\prime \subset T^\prime$, where $M^\prime$ is a copy of $\mu_2$. 
    Thus, we get 
    \[
    H \defeq C_K(M^\prime)^0
    \]
    as a copy of $\Spin_8$ inside of $\Spin_9= K \subset G$, having as roots exactly the long roots of $\Lie G$, and we are done.
\end{proof}

\begin{lemma}
\label{reductive_centralizer}
    Let $M \subset T \subset G$ with $G$ reductive and $T$ a maximal torus of $G$.\\
    Then the identity component of the centralizer $Z \defeq C_G(M)^0$ is smooth and reductive.
\end{lemma}

Let us mention that \Cref{reductive_centralizer} is a particular case of \cite[Proposition A.8.12]{CGP15}. We provide a direct, elementary proof below.

\begin{proof}
Existence and smoothness are general facts since $M$ is linearly reductive.\\
    Next, let us assume that $Z$ is not reductive: since $Z$ is contained in $T$, there is some root $\gamma$ of $\Lie G$ such that 
    \[
    \mathfrak{g}_\gamma \subset \Lie U, \quad \text{where } U \defeq R_u(Z).
    \]
Since $M$ is linearly reductive, we have that $\Lie Z$ is the fixed point subalgebra $(\Lie G)^M$. In particular, if the root subspace associated to $\gamma$ is contained in $\Lie Z$, then the same holds for $-\gamma$. Moreover, $U$ being normal in $Z$ implies that $\Lie U$ is a $p$-Lie ideal of $\Lie Z$, hence
\[
[\mathfrak{g}_\gamma,\mathfrak{g}_{-\gamma}] \subset \Lie U,
\]
which is absurd, because the above bracket is nonzero and contained in the Lie algebra of the maximal torus $T$.
\end{proof}

\subsection{Type $G_2$}
Let us consider a group of type $G_2$: we begin with the case of characteristic $3$ because the edge hypothesis is satisfied, and we once again get that all parabolic subgroups are of quasi-standard type. Then we move on to characteristic $2$, where a more exotic behavior takes place.

\subsubsection{Characteristic three}

Let $G$ be of type $G_2$ in characteristic $3$. 

\begin{proposition}
    Let $P \subset G$ be a parabolic subgroup with \[
    P_{\text{red}}= B = P^{\alpha_1}\cap P^{\alpha_2}.
    \]
    Then $P$ is the pull-back of a parabolic subgroup of \emph{standard type} by an isogeny with no central factor; in particular, (\ref{inclusion0}) is an equality.
\end{proposition}

\begin{proof}
    The first step consists in taking the quotient of $P$ by the kernel of the iterated Frobenius of highest possible height. Hence we can make the hypothesis that the Frobenius kernel $G^1$ is not contained in $P$, which is equivalent to assuming that the Lie algebra of $P$ is not the whole of $\Lie G$. 
    Looking at structure constants, we see that
    \[
    [\mathfrak{g}_{-3\alpha_1-2\alpha_2}, \mathfrak{g}_{\gamma}] = \mathfrak{g}_{-3\alpha_1-2\alpha_2+\gamma},
    \]
    for any positive root $\gamma$ such that $-3\alpha_1-2\alpha_2+\gamma$ is still a root. Hence, if $\mathfrak{g}_{-3\alpha_1-2\alpha_2}$ intersects $\Lie P$, then all other negative root subspaces do and we get $\Lie P = \Lie G$. Thus by our assumption, we necessarily have
    \begin{align}
    \label{no_longest}
        \mathfrak{g}_{-3\alpha_1-2\alpha_2} \cap \Lie P = 0.
    \end{align}
    Moreover, by taking the quotient via the very special isogeny $\pi$, which exists because we are in characteristic $3$, we can assume that $N\defeq N_G$ is not contained in $P$. In other words, we make the hypothesis that at least one root subspace associated to a short negative root does not intersect $\Lie P$. Let us notice that 
    \[
    [\mathfrak{g}_{-2\alpha_1-\alpha_2}, \mathfrak{g}_{\alpha_1}] = \mathfrak{g}_{-\alpha_1-\alpha_2}, \quad
    [\mathfrak{g}_{-\alpha_1-\alpha_2}, \mathfrak{g}_{\alpha_2}] = \mathfrak{g}_{-\alpha_1}, \quad
    [\mathfrak{g}_{-\alpha_1}, \mathfrak{g}_{-\alpha_1-\alpha_2}] = \mathfrak{g}_{-2\alpha_1-\alpha_2}, 
    \]
    because all three structure constants are equal to $\pm 2$. Thus, if the root subspace associated to $-2\alpha_1$ or to $\alpha_1-\alpha_2$ intersects $\Lie P$, then $\Lie N$ is contained in $\Lie P$ and we have a contradiction. Hence we are in the following situation, where $l,r,s$ are non-negative integers and we place next to each negative root $\gamma$ the height of the intersection $P \cap U_{\gamma}$.

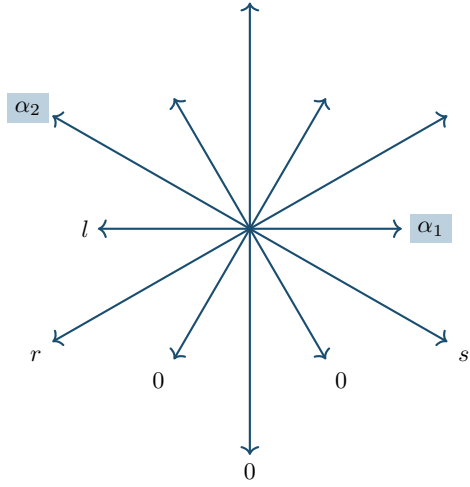
    \begin{center}
\begin{tikzpicture}
    \foreach\ang in {60,120,...,360}{
     \draw[->,cite!80!black,thick] (0,0) -- (\ang:2cm);
    }
    \foreach\ang in {30,90,...,330}{
     \draw[->,cite!80!black,thick] (0,0) -- (\ang:3cm);
    }
    \node[anchor= west,scale=0.9] at (2,0) {\colorbox{cite!30!white}{$\alpha_1$}};
    \node[anchor= east,scale=0.9] at (-2,0) {$l$};
    \node[anchor=south west,scale=0.9] at (-3.3,1.3) {\colorbox{cite!30!white}{$\alpha_2$}};
    \node[anchor=north east,scale=0.9] at (3,-1.5) {$s$};
    \node[anchor=north west,scale=0.9] at (-3,-1.5) {$r$};
    \node[anchor=north,scale=0.9] at (-1.2,-1.8) {$0$};
    \node[anchor=north,scale=0.9] at (1.2,-1.8) {$0$}; 
    \node[anchor=north,scale=0.9] at (0,-3) {$0$};
  \end{tikzpicture}
\end{center}
Finally, we can also say that $r$ must be equal to zero, because
\[
[\mathfrak{g}_{-3\alpha_1-\alpha_2},\mathfrak{g}_{\alpha_1}]  = \mathfrak{g}_{-2\alpha_1-\alpha_2}.
\]
This means that 
\[P = G^l P^{\alpha_1} \cap G^s P^{\alpha_2},
\]
with one among $l$ and $s$ which vanishes (because otherwise $P$ would contain $G^1$).
\end{proof}

\subsubsection{Characteristic two}
\label{exotic}
Let $G$ be of type $G_2$  in characteristic $2$. 
Let us briefly recall some results which are shown in \cite[Section 2.6]{Maccan}. First, there exist two maximal $p$-Lie subalgebras of $\Lie G$ containing $\Lie P^{\alpha_1}$, namely:
\[
\mathfrak{h} \defeq \Lie P^{\alpha_1} \oplus \mathfrak{g}_{-2\alpha_1-\alpha_2} \quad \text {and} \quad
\mathfrak{l} \defeq \Lie P^{\alpha_1} \oplus \mathfrak{g}_{-\alpha_1} \oplus \mathfrak{g}_{-\alpha_1-\alpha_2}.
\]
\begin{lemma}
\label{h_l}
 The $p$-Lie subalgebras of $\Lie G$ containing strictly $\Lie P^{\alpha_1}$ are exactly $\mathfrak{h}$ and $\mathfrak{l}$.
\end{lemma}

Next, we consider the subgroups $H$ and $L$ of the group $G$, defined as being of height one with Lie algebra respectively equal to
\begin{align}
\label{defHL}
\Lie H \defeq \mathfrak{g}_{-2\alpha_1-\alpha_2} \quad \text{and} \quad \Lie L \defeq \mathfrak{g}_{-\alpha_1} \oplus \mathfrak{g}_{-\alpha_1-\alpha_2}.
\end{align}
Finally, set
\[
P_\mathfrak{h} \defeq \langle H, P^{\alpha_1} \rangle \quad \text{and} \quad P_\mathfrak{l} \defeq \langle L, P^{\alpha_1} \rangle.
\]
This defines two parabolic subgroups which cannot be described as $(\ker \xi)P^{\alpha_1}$ for some isogeny $\xi$ with source $G$, due to the fact that $\Lie G$ is simple (see \cite[4.4]{Strade}).\\
These two exotic subgroups are enough to complete the classification in type $G_2$, as follows.

\begin{proposition}
\label{rank1_G}
    Let $G$ be of type $G_2$ in characteristic two.\\
    Then all parabolic subgroups of $G$ having $P^{\alpha_1}$ as reduced part are either of standard type, or obtained from $P_{\mathfrak{l}}$ and $P_{\mathfrak{h}}$ by pulling back with an iterated Frobenius homomorphism.
\end{proposition}

\begin{proof}
    See \cite[Proposition 3]{Maccan}.    
\end{proof}

\begin{proposition}
    Let $P \subset G$ be a parabolic subgroup with \[
    P_{\text{red}}= B = P^{\alpha_1}\cap P^{\alpha_2}.\]
    Then (\ref{inclusion0}) is an equality.
\end{proposition}

\begin{proof}
Taking the quotient of $P$ by the highest possible Frobenius kernel allows us to assume (\ref{no_longest}), as in the previous proof. The situation is summarized below, where $l,m,n,r,s$ are non-negative integers on which we now determine some conditions.

    
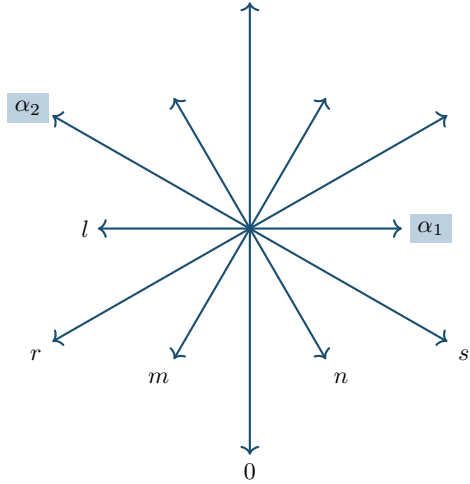
\begin{center}
\begin{tikzpicture}
    \foreach\ang in {60,120,...,360}{
     \draw[->,cite!80!black,thick] (0,0) -- (\ang:2cm);
    }
    \foreach\ang in {30,90,...,330}{
     \draw[->,cite!80!black,thick] (0,0) -- (\ang:3cm);
    }
    \node[anchor= west,scale=0.9] at (2,0) {\colorbox{cite!30!white}{$\alpha_1$}};
    \node[anchor= east,scale=0.9] at (-2,0) {$l$};
    \node[anchor=south west,scale=0.9] at (-3.3,1.3) {\colorbox{cite!30!white}{$\alpha_2$}};
    \node[anchor=north east,scale=0.9] at (3,-1.5) {$s$};
    \node[anchor=north west,scale=0.9] at (-3,-1.5) {$r$};
    \node[anchor=north,scale=0.9] at (-1.2,-1.8) {$m$};
    \node[anchor=north,scale=0.9] at (1.2,-1.8) {$n$}; 
    \node[anchor=north,scale=0.9] at (0,-3) {$0$};
  \end{tikzpicture}
\end{center}
First, we notice that
\[
[\mathfrak{g}_{-3\alpha_1-\alpha_2},\mathfrak{g}_{-\alpha_2}] = [\mathfrak{g}_{-2\alpha_1-\alpha_2},\mathfrak{g}_{-\alpha_1-\alpha_2}] = \mathfrak{g}_{-3\alpha_1-2\alpha_2}
\]
because the two structure constants are respectively $\pm 1$ and $\pm 3$. Thus, one among $r$ and $s$ must vanish, and analogously for $m$ and $n$.\\\\
$\bullet$ Assume that $r$ vanishes and $n$ does not, which yields that $m$ is equal to zero. Moreover, $l$ is nonzero because
\[
[\mathfrak{g}_{-\alpha_1-\alpha_2}, \mathfrak{g}_{\alpha_2}] = \mathfrak{g}_{-\alpha_1}.
\]
A direct computation, involving the root subgroups computed in \cite[Appendix]{Maccan}, shows that for $a,b \in \Ga$,
\begin{align}
\label{n=1}
(u_{-\alpha_1}(a), u_{-\alpha_1-\alpha_2}(b)) = u_{-3\alpha_1-\alpha_2} (a^2b) \, u_{-3\alpha_1-2\alpha_2} (ab^2).
\end{align}
By \cite[Proposition $8$]{Wenzel}, if $l$ is bigger than $1$, then the unipotent infinitesimal part $U_P^-$, whose definition is recalled in (\ref{infinitesimal}), has nontrivial intersection with $U_{-3\alpha_1-\alpha_2}$. This yields that $r$ is nonzero, which contradicts our assumption. On the other hand, if $n \geq 2$, then $P$ has nontrivial intersection with $U_{-3\alpha_1-2\alpha_2}$, contradicting (\ref{no_longest}). This means that $l=n=1$, thus for any $s$ we get
\[
P = P_{\mathfrak{l}} \cap G^s P^{\alpha_2}.
\]
$\bullet$ Next, assume that both $r$ and $n$ vanish. If $m$ is also zero, then
\[
P = G^l P^{\alpha_1} \cap P^{\alpha_2} \quad \text{or} \quad P= P^{\alpha_1} \cap G^s P^{\alpha_2}
\]
which are standard. Thus we can assume $m$ to be nonzero. Another direct computation shows that for $b \in \Ga$,
\begin{align}
\label{m=1}
( u_{\alpha_1}(1), u_{-2\alpha_1-\alpha_2}(b)) = u_{\alpha_2} (b^2) \, u_{-3\alpha_1-2\alpha_2} (b),
\end{align}
hence if $m\geq 2$, then $P$ has again nontrivial intersection with $U_{-3\alpha_1-2\alpha_2}$. Moreover, $l$ must vanish too, because 
\[
[\mathfrak{g}_{-\alpha_1},\mathfrak{g}_{-2\alpha_1-\alpha_2}] = \mathfrak{g}_{-3\alpha_1-\alpha_2}
\]
hence if $l$ does not vanish then the same holds for $r$. This gives for any $s$,
\[
P = P_{\mathfrak{h}} \cap G^s P^{\alpha_2}.
\]
$\bullet$ We are left with the case where $r$ is nonzero, which implies that $s$ must vanish. From the equality 
\[
[\mathfrak{g}_{-\alpha_1-\alpha_2}, \mathfrak{g}_{\alpha_1}] = \mathfrak{g}_{-\alpha_2},
\]
we deduce that $n$ is also equal to zero. 
Finally, a direct computation shows
\[
(u_{-3\alpha_1-\alpha_2}(a), u_{\alpha_1}(b)) = u_{-2\alpha_1-\alpha_2} (a) \, u_{-\alpha_1-\alpha_2}(b) \, u_{-\alpha_2}(b).
\]
By \cite[Proposition $8$]{Wenzel}, if $r$ is nonzero then we have that $u_{-\alpha_2}(b)$ belongs to $P$ for $b \in \boldsymbol{\alpha}_2$; however this is a contradiction with the fact that $s$ is equal to zero.
\end{proof}


\section{Some geometric consequences}
Let us start by mentioning one immediate consequence of \Cref{mainnnn}. The statement, together with the case of maximal reduced part (\Cref{kerphi}) implies the following.

\begin{corollary}
    Except for a group of type $G_2$ in characteristic $2$, any parabolic subgroup is of quasi-standard type.
\end{corollary}
For the $G_2$ case, two \emph{exotic} parabolic subgroups do occur: see \Cref{exotic} for more details.\\\\
Let us move on to investigating some geometric properties of rational projective homogeneous spaces. In order to do it, we make use of the family of contractions of maximal relative Picard rank with source $X=G/P$, denoted
\[
f_\alpha \colon X \longrightarrow G/Q^\alpha = G/\langle P,P^\alpha\rangle, \quad \text{for } \alpha \in \Delta \backslash I,
\]
which are defined in (\ref{f_alpha}). Those morphisms are entirely determined, up to a permutation, by the variety $X$: see \cite[Section 3.2]{Maccan} for more details.

\begin{remark}
Considering the underlying varieties, \Cref{mainnnn} means that the product of the contractions $f_\alpha$ realises any rational homogeneous space $X$ as a closed subvariety of the product
\begin{align}
\label{emb}
    X = G/P \longhookrightarrow \prod_{\alpha \in \Delta \backslash I} G/Q^\alpha 
\end{align}
By \cite[Theorem 1]{Maccan}, each $G/Q^\alpha$ is either isomorphic to a flag variety having as stabilizer a maximal reduced parabolic (hence defined over $\Z$) or, in the case of $G_2$ and characteristic $2$, it can be isomorphic to the variety
\[
\mathcal{X} \defeq G_2/P_{\mathfrak{l}}.
\]
The latter is described in \cite[Section 2.6.4]{Maccan} as being a general hyperplane section of the $\Sp_6$-homogeneous variety of isotropic $3$-dimensional subspaces in a $6$-dimensional vector space.
\end{remark}

\begin{corollary}
\label{veryample}
    Every ample line bundle on $X = G/P$ is very ample.
\end{corollary}

Before proving this statement, let us briefly recall what ampleness on $X$ looks like: thanks to the Białynicki-Birula decomposition, there is an explicit basis of the Picard group of $X$, given by the Schubert divisors $D_\alpha$ defined in (\ref{Dalpha}). In \cite[Section 3.2]{Maccan}, we show that a line bundle on $X$ is ample if and only if it writes as a linear combination of the $D_\alpha$s with strictly positive coefficients. In particular, from \Cref{veryample} we can deduce that $X$ has a minimal embedding into projective space as
\[
X \longhookrightarrow \proj (H^0(X, D)^\vee), \quad D \defeq \sum_{\alpha \in \Delta \backslash I} D_\alpha.
\]

\begin{proof} (of \Cref{veryample})
    When the stabiliser $P$ is reduced this is a well-known fact; it is also true for the exotic variety $\mathcal{X}$ above, due to its construction as a general hyperplane section of a generalised $\Sp_6$-flag variety. Thus, it holds for any rational projective homogeneous variety of Picard rank one. The embedding (\ref{emb}) into the product of the $G/Q^\alpha$ is enough to conclude the analogous statement for any $X$.
\end{proof}


\subsection{Stabiliser of contractions: standard type}

Let us now compute the stabilizer $Q^\alpha$ of each contraction $f_\alpha$ of maximal relative Picard rank, in the case of a parabolic subgroup $P$ of standard type. This shows that parabolics of quasi-standard type - or even exotic ones in type $G_2$ - can already occur in this context.\\

Let $P$ be of standard type and $r$ be the Picard rank of $G/P$; then there are distinct simple roots $\beta_1,\ldots, \beta_r$ and non-negative integers $m_1\leq \ldots \leq m_r$ satisfying 
\begin{align}
\label{standard}
P = G^{m_1 }P^{\beta_1} \cap \ldots \cap G^{m_r} P^{\beta_r}.
\end{align}
Considering the quotient of $G$ by the $m_1$-th iterated Frobenius kernel allows us to assume that $m_1$ is equal to one, so that the contraction $f_{\beta_1}$ associated to $\beta_1$ 
 is smooth. Let us denote as 
\[Q^i \defeq Q^{\beta_i} = \langle P, P^{\beta_i} \rangle
\]
the stabilizer of the target of the contraction associated to $\beta_i$; in particular,
\[
Q^1 = P^{\beta_1}.
\]
Clearly, by definition $Q^i$ is contained in $G^{m_i}P^{\beta_i}$, however it could a priori be smaller; by \Cref{lem:m_alpha}, the intersection of $Q^i$ with $U_{-\beta_i}$ necessarily has height $m_i$. We examine the different situations that can happen, in order to determine under which conditions $Q^i$ is not of standard type anymore. First, let us introduce the following notation in type $G_2$.

\begin{definition}
    Let $G$ be of type $G_2$ in characteristic two.\\
    The pullback by an $m$-th iterated Frobenius morphisms of the height one subgroup $L$, whose definition is recalled in (\ref{defHL}), is denoted
    \[
    L^m \defeq (F^m)^{-1}L.
    \]
\end{definition}

\begin{lemma}
\label{Qforstandard}
    Let $P$ be a parabolic subgroup of standard type as in (\ref{standard}), with $m_1$ equal to one.
    \begin{enumerate}[(1)]
        \item If $G$ does not satisfy the edge hypothesis and is not of type $G_2$, then each $Q^i$ is standard.
        \item If $G$ satisfies the edge hypothesis, 
        \[
        Q^i = N^{m_i-1} P^{\beta_i}
        \]
        if and only if all positive long roots containing $\beta_i$ in their support also contain some $\beta_l$ with $m_l < m_i$.
        \item If $G$ is of type $G_2$ in characteristic $2$, the only case where a parabolic not of standard type (with maximal reduced part) appears is when
        \[
        P = P^{\alpha_2} \cap G^m P^{\alpha_1},
        \]
        for which we get one exotic stabiliser, namely
        \[
        Q^1 = L^{m-1} P^{\alpha_1}.
        \]
   \end{enumerate}
\end{lemma}

\begin{proof}
\textbf{(1)} Assume $G$ is not of type $G_2$ and does not satisfy the edge hypothesis.\\
Then by \Cref{kerphi}, the only parabolic subgroup scheme with reduced part $P^{\beta_i}$ and whose intersection with $U_{-\beta_i}$ has height $m_i$ is 
\[
G^{m_i}P^{\beta_i},
\]
so the latter necessarily coincides with $Q^i$.\\
\textbf{(2)} Assume $G$ satisfies the edge hypothesis. By \Cref{lem:m_alpha} together with \Cref{kerphi}, $Q^i$ is obtained from its reduced part either by fattening with $G^{m_i}$ or with $N^{m_i-1}$. The second case can happen if and only if the equality 
\begin{align}
\label{uguaglianza}
(P = ) \quad G^{m_i} P^{\beta_i} \cap \left(  \bigcap_{j\neq i} G^{m_j}P^{\beta_j} \right) = N^{m_i-1} P^{\beta_i} \cap \left(  \bigcap_{j\neq i} G^{m_j}P^{\beta_j} \right)  
\end{align}
is satisfied. Let $\varphi$ and $\psi$ be respectively the function associated to the left and the right hand term. 
If the support of $\gamma$ does not contain $\beta_i$, then the whole root subgroup $U_{-\gamma}$ is contained in $P^{\beta_i}$, hence 
\[
\varphi(\gamma) = \psi(\gamma)= \infty.
\]
Thus we can assume that $\beta_i$ is in the support of $\gamma$.\\
If $\gamma$ also contains some $\beta_l$ in its support satisfying $m_l < m_i$, then the two functions still coincide on $\gamma$. Finally, assume that the support of $\gamma$ only contains simple roots $\beta_l \in \Delta \backslash I$ with $m_l \geq m_i$. If $\gamma$ is short, then 
\[
N^{m_i-1}\cap U_{-\gamma}
\]
has height equal to $m_i$, hence $\varphi(\gamma) = \psi(\gamma)$ once again. On the other hand, if $\gamma$ is long then the intersection above 
has height $m_i-1$: this is the only case where the equality (\ref{uguaglianza}) cannot hold, because 
\[
\varphi(\gamma) = m_i \quad \text{while} \quad \psi(\gamma) = m_i-1.\]
Summarizing, the parabolic subgroup $Q^i$ is \emph{not} of standard type, and is in particular equal to 
\[N^{m_i-1}P^{\beta_i},\]
if and only if all positive long roots containing $\beta_i$ in their support also contain some $\beta_l$ with $m_l < m_i$.\\
%
\textbf{(3)} The last case to look at is when $G$ is of type $G_2$, the characteristic is $p=2$ and the Picard rank of $G/P$ is equal to $2$.\\
Let us assume that
\[
P = P^{\alpha_2}\cap G^mP^{\alpha_1}
\]
for some positive $m$; then by \Cref{lem:m_alpha} the intersection $Q^1 \cap U_{-\alpha_1}$ has height $m$, which means that the root subspace associated to $-\alpha_1$ is contained in $\Lie Q^1$. Hence by \Cref{h_l}, the latter must contain the Lie subalgebra $\mathfrak{l}$. This fact, together with the classification of parabolic subgroups with reduced part $P^{\alpha_1}$, given in \Cref{rank1_G}, implies that $L^{m-1} P^{\alpha_1}$ is contained in $Q^1$. Moreover, the equality 
\[
P = P^{\alpha_2} \cap L^{m-1} P^{\alpha_1}\]
holds, because intersecting both sides with $U_{-\gamma}$ for any positive root $\gamma \neq \alpha_1$ gives a trivial intersection. This shows that 
\[
Q^1= L^{m-1} P^{\alpha_1}.
\]
Thus, in this case we obtain a parabolic which is not of quasi-standard type.\\
On the other hand, if 
\[
P= P^{\alpha_1} \cap G^m P^{\alpha_2}
\]
for some $m$, then the intersection $Q^2 \cap U_{-\alpha_2}$ has height $m$, hence by \Cref{kerphi} we have necessarily that 
\[Q^2 = G^m P^{\alpha_2}\]
is standard.
\end{proof}

\begin{example}
    Assume that the rank $r$ is at least equal to two, that the integer $m_r$ is nonzero and that the simple root $\beta_r$ is long. Then $Q^r = G^{m_r}P^{\beta_r}$ is of standard type.
\end{example}
\begin{example}
    A few parabolic subgroups in rank two, with 
    \[P_{\text{red}}= P^\alpha \cap P^\beta,\]
    are listed below: in particular, these examples underline the importance of paying attention to the duality between the groups of type $B_n$ and $C_n$.
\begin{center}
\begin{tabular}{ |c|c|c|c|}
\hline
$P$ & type & $Q^\alpha$ & $Q^\beta$ \\
\hline
$P^{\alpha_{n-1}}\cap G^mP^{\alpha_n}$ & $B_n$ & $Q^{n-1}= P^{\alpha_{n-1}}$ & $Q^n= N^{m-1}P^{\alpha_n}$ \\
\hline
$P^{\alpha_1}\cap G^mP^{\alpha_n}$ & $C_n$ & $Q^1= P^{\alpha_1}$ & $Q^n = G^mP^{\alpha_n}$\\
\hline
$P^{\alpha_n} \cap G^mP^{\alpha_1}$ & $C_n$ & $Q^n= P^{\alpha_n}$ & $Q^1 = N^{m-1}P^{\alpha_1}$\\
\hline
$P^{\alpha_n} \cap G^mP^{\alpha_1}$ & $B_n$ & $Q^n= P^{\alpha_n}$ & $Q^1= G^mP^{\alpha_1}$ \\
\hline
\end{tabular}
\end{center}
\end{example}

\vskip 10 pt
\subsection{Existence of smooth contractions}
We address the question of whether a rational projective homogeneous variety admits a smooth contraction. We obtain in \Cref{fibrations} a structure result, saying that such a variety can be obtained by iterated Zariski locally trivial fibrations whose targets are flag varieties of Picard rank one. Let us start by considering morphisms of maximal relative Picard rank. 

\begin{lemma}
\label{nokernel}
    Let $X$ be a rational projective homogeneous variety. Then there is a semisimple, simply connected group $G$ and a parabolic subgroup $P$  of $G$, satisfying
    \[
    X=G/P,
    \] 
    such that for each simple factor of $G$, its intersection with $P$ does not contain the kernel of any isogeny with no central factor. 
\end{lemma}

\begin{proof}
    Let us write the variety $X$ as
    \[
    X= G^\prime /P^\prime
    \]
     for some semisimple and simply connected group $G^\prime$ and a parabolic subgroup $P^\prime$ of $G^\prime$. Let $H$ be the subgroup of $G^\prime$ generated by all normal noncentral subgroups of height one of $G^\prime$ which are contained in $P^\prime$. Then we can write
     \[
     X = (G^\prime/H) / (P^\prime/H) =: G/P.
     \]
     The group $G$ is still simply connected and the parabolic subgroup $P$ is as wanted.
\end{proof}
The question of the existence of a smooth contraction is now related to finding some simple root $\alpha$, not belonging to the Levi subgroup of the reduced part of $P$, such that the stabiliser  $Q^\alpha$ is reduced (hence, smooth).

\begin{remark}
\label{betaone}
If $P^\prime$ is a parabolic subgroup of quasi-standard type of $G^\prime$, then \Cref{nokernel} can be illustrated in a simpler way as follows. Let 
\[
    P^\prime = \bigcap_{i=1}^r (\ker \xi_i) P^{\beta_i},
    \]
    where $r$ is the Picard rank of $X$ and the $\xi_i$ are isogenies with no central factor, uniquely determined as in 
 \Cref{kerphi}. Moreover, let us order them such that
 \[
 \ker \xi_1 \subseteq \ldots \subseteq \ker \xi_r,
 \]
 which can be done thanks to \Cref{minimale}. Then the subgroup $H$ of $G$ generated by all normal noncentral subgroups of height one contained in $P$ is 
 \[
 H = \ker \xi_1.
 \]
\end{remark}

\begin{proposition}
    Let $G$ be simply connected and $X=G/P$ be such that $P$ contains no kernel of isogenies with no central factor. \\
    Then there is a simple root $\alpha$ not belonging to the Levi subgroup of $P_{\text{red}}$ such that
    \[
    Q^\alpha = P^\alpha
    \]
    if and only if $P$ is of quasi-standard type.
\end{proposition}





\begin{proof}
\textbf{(1)}: Let us assume that $P$ is of quasi-standard type (which we recall is always the case except for a group of type $G_2$ in characteristic two). 
By \Cref{betaone}, we can write
\[
P = P^{\alpha} \cap (\ker \xi)P^\prime
\]
for some simple root $\alpha$ and an isogeny $\xi$ with no central factor. Then $Q^\alpha$ is equal to $P^\alpha$, which in particular means that
\[
f_\alpha \colon X = G/P \longrightarrow G/P^\alpha
\]
is a smooth contraction.\\
\textbf{(2)}: Let us assume that $G$ is of type $G_2$, that $p=2$ and that $P$ is not of quasi-standard type. Since by assumption $P$ does not contain the Frobenius kernel of $G$, it must be of the form
    \[
    P = P_{\mathfrak{h}} \cap G^s P^{\alpha_2} \quad \text{or} \quad P = P_{\mathfrak{l}} \cap G^s P^{\alpha_2}
    \]
    for some non-negative integer $s$. In both cases, we have that the height of the intersection of $Q^2$ and $U_{-\alpha_2}$ is equal to $s$.
    Hence 
    \[
    Q^2 = G^s P^{\alpha_2}
    \]
    is smooth if and only if $s$ is zero: this cannot happen because
    \[
    P = P_{\mathfrak{h}} \cap P^{\alpha_2} = P^{\alpha_1}\cap P^{\alpha_2}=  B \quad \text{and} \quad P = P_{\mathfrak{l}} \cap P^{\alpha_2} = G^1 P^{\alpha_1} \cap P^{\alpha_2}
    \]
    are both of standard type, which contradicts our assumption.\\
    Thus we have $s\geq1$. Let us consider the subgroup $Q^1$: in the first case, the height of the intersection of $Q^1$ and $U_{-2\alpha_1-\alpha_2}$ is equal to one, 
    which implies that $\Lie Q^1$ contains $\mathfrak{h}$ and finally that
    \[
    Q^1 = P_{\mathfrak{h}}
    \]
    is nonreduced. Analogously, in the second case the height of the intersection of $Q^1$ and $U_{-\alpha_1}$ is equal to one, 
    which implies that $\Lie Q^1$ contains $\mathfrak{l}$ and finally that
    \[
    Q^1 = P_{\mathfrak{l}}
    \]
    is again non-reduced.
\end{proof}

\begin{corollary}
\label{cor}
    Over an algebraically closed field of characteristic $p=2$, the isomorphism classes of $G_2$-homogeneous varieties with Picard group of rank $2$ are in one-to-one correspondence with the following parabolic subgroups:
    \begin{enumerate}[(a)]
        \item $P^{\alpha_1}\cap P^{\alpha_2}= B$;
        \item $G^m P^{\alpha_1} \cap P^{\alpha_2}$, for $m\geq 1$;
        \item $P^{\alpha_1} \cap G^m P^{\alpha}$, for $m\geq 1$;
        \item $P_{\mathfrak{h}} \cap G^m P^{\alpha_2}$, for $m \geq 1$;
        \item $P_{\mathfrak{l}}\cap G^m P^{\alpha_2}$, for $m\geq 1$.
    \end{enumerate}
\end{corollary}

\begin{remark}
     In the above list, $(a)$, $(b)$ and $(c)$ are of quasi-standard type, while $(d)$ and $(e)$ are not. This leads to different geometric properties of the homogeneous spaces: for instance, the variety with stabiliser $(a)$ has two smooth contractions, those in cases $(b)$ and $(c)$ have one, while those in cases $(d)$ and $(e)$ have none; see \Cref{fibrations} below.
\end{remark}

Before proving \Cref{cor}, let us recall what we mean by automorphism group and recall the statement of Blanchard's Lemma; for the latter, see \cite[$7.2$]{Brion17}.\\
For a proper algebraic scheme $X$ over a perfect field $k$, the functor 
\[
\underline{\Aut}_X \colon (\mathbf{Sch}/k)_{\text{red}} \longrightarrow \mathbf{Grp}, \quad T \longmapsto \Aut_T(X_T),
\]
sending a reduced $k$-scheme $T$ to the group of automorphisms of $T$-schemes of $X\times_k T$, is represented by a reduced group scheme $\underline{\Aut}_X$ which is locally of finite type over $k$: see \cite[Theorem 3.6]{MatsumuraOort}. 
We denote as \[\underline{\Aut}_X^0\] its identity component, which is a smooth connected algebraic group.

\begin{theorem}[Blanchard's Lemma]
\label{blanchard}
Let $f \colon X \rightarrow Y$ be a contraction between projective varieties over $k$. Assume 
that $X$ is equipped with an action of a connected algebraic group $G$. Then there exists a unique $G$-action on $Y$ such that the morphism $f$ is $G$-equivariant.
\end{theorem}

\begin{proof} (of \Cref{cor}). 
    Let us write
    \[
    X = G/P = G^\prime /P^\prime,
    \]
    where $G$ is of type $G_2$, $P$ is a parabolic subgroup with reduced part equal to the Borel $B$, and $G^\prime$ is an adjoint semisimple group, such that both actions on $X$ are faithful. By \Cref{blanchard} applied to the $G^\prime$-action, the contraction
    \[
    f_{\alpha_2} \colon X \longrightarrow G/Q^{\alpha_2} \simeq G/P^{\alpha_2} =:Y
    \]
    induces an action of $G^\prime$ onto $Y$. 
    By \cite[Theorem 1]{Demazure}, the reduced automorphism group
    \[
     \underline{\Aut}_Y^0
     \]
     of $Y$ is $G$; in particular, the only semisimple group which can act faithfully on $Y$ is the group $G$, which implies $G^\prime= G$. 
    We can thus deduce that $P$ and $P^\prime$ are conjugated in $G$, because the group of type $G_2$ has no outer automorphisms. Thanks to our classification of parabolic subgroups, we get the desired list above.
\end{proof}


We now associate to any homogeneous space $X$ - with stabiliser of quasi-standard type - a canonical fibration which, whenever the stabiliser of $X$ is non-reduced, is realised as the smooth contraction with minimal relative Picard rank.\\
Here we focus on isomorphism classes of varieties rather than on conjugacy classes of parabolic subgroups, so we start by applying \Cref{nokernel}. Then we repeat this construction to build from $X$ a finite sequence of Zariski locally trivial contractions, whose targets are generalised flag varieties of Picard rank one. Let us start by noticing that the quotient morphism
\[
G \longrightarrow G/P
\]
is Zariski locally trivial (and in particular, smooth) if and only if the parabolic subgroup is reduced.

\begin{remark}
    Let 
    \[
    X=G/P
    \]
    with $P$ a parabolic subgroup of quasi-standard, such that for each simple factor of $G$, its intersection with $P$ does not contain the kernel of any isogeny with no central factor. Then there is a unique minimal reduced parabolic subgroup of $G$ containing $P$, given by
    \[
    P^{\sm} = \bigcap \{P^\alpha \colon \alpha \in \Delta \backslash I \, \text{ and } \, Q^\alpha=P^\alpha\}.
    \]
    Moreover, $P$ can be written in a unique way as 
    \begin{align}
    \label{intersezione}
    P = P^{\sm} \cap (\ker \xi) P^\prime,
    \end{align}
    where $\xi$ is an isogeny with no central factor and $P^\prime$ is a parabolic subgroup with reduced part
    \[
    P^\prime_{\text{red}} = \bigcap \{P^\alpha \colon \alpha \in \Delta \backslash I \, \text{ and } \, Q^\alpha \neq P^\alpha\},
    \]
    such that $P^\prime$ does not contain the kernel of any isogeny with no central factor.
\end{remark}

\begin{proposition}
\label{fibrations}
    Let $X$ be a homogeneous projective variety of Picard rank $r\geq 2$ whose automorphism group has no $G_2$ factor if the characteristic is $2$.\\ 
    There is a finite sequence of Zariski locally trivial contraction morphisms
    \[
    g_s \colon X_s \longrightarrow Y_{s}, \quad 1\leq s \leq r
    \]
    such that $X_1=X$, each $Y_s$ is a rational projective homogeneous variety of Picard rank one and $X_{s+1}$ is the fiber of $g_s$.
\end{proposition}

\begin{proof}
    Let us write 
    \[
    X=G/P
    \]
    with $P$ a parabolic subgroup of quasi-standard type; thanks to \Cref{nokernel}, assume that there is some simple root $\alpha$ such that $Q^\alpha$ is smooth.  
    Then the contraction morphism
    \[
    f_\alpha \colon X \longrightarrow G/P^\alpha =: Y_1
    \]
    is a Zariski locally trivial contraction, with smooth, connected and projective fiber equal to 
    \[
    X^\prime \defeq P^\alpha /P,
    \]
    The idea is now to replace $X$ by $X^\prime$; in order to do this, let us consider the image 
    \[
    G^{\alpha} \defeq \mathrm{im} (P^\alpha \longrightarrow \underline{\Aut}_{X_1})
    \]
    which is a semisimple adjoint group. Let $\overline{P}$ be the image of $P$, then let $G^\prime$ and $P^\prime$ be respectively the simply connected cover of $G^\alpha$ and the preimage of $\overline{P}$ in $G^\prime$. Then we have
    \[
    X^\prime = G^\alpha/\overline{P} = G^\prime/P^\prime.
    \]
    This allows to start again by replacing the variety $X$ by $X^\prime$; since the Picard rank of the fiber decreases by one at each step, this process terminates in $r$ steps.
    \end{proof}

    Let us mention that if $p>3$ then for any parabolic subgroup $P$ the Chow motives of $G/P$ and of $G/P_{\text{red}}$ are isomorphic, as proven in \cite[Theorem 1.3]{Srinivasan}; we wonder whether this result could be recovered from \Cref{fibrations}.


\subsection{Fano homogeneous spaces}
Over an algebraically closed field of characteristic zero, every $G/P$ is Fano; moreover, there are only a finite number with fixed dimensions.\\
Both of these statements prove false in positive characteristics.

\begin{example}
\label{esempio}
Let us consider the following incidence varieties in the product of two projective planes:
\[
X_m \defeq \{x_0^{p^m} y_0+x_1^{p^m}y_1+x_2^{p^m}y_2 =0\} \subset \proj^2 \times \proj^2,
\]
where $m \geq 0$. The group $\SL_3$ acts on the first factor with a twisted action via an $m$-th iterated Frobenius morphism, and with the standard action on the second factor, thus preserving $X_m$. These form an infinite family of non-isomorphic rational projective homogeneous spaces of dimension $3$, which are not Fano if $p^m >3$.
\end{example}

Nevertheless, the following finiteness property still holds: only a finite number of homogeneous spaces of fixed dimension are Fano varieties. The idea of the proof is that, except for a finite number of cases, there is at least one \emph{incidence relation} in the embedding
\[
f \colon X=G/P \longhookrightarrow \prod_{\alpha \in \Delta\backslash I} G/Q^\alpha
\]
which is \emph{twisted too much} via the kernel of an isogeny with no central factor.\\
Since $G$ is simply connected, any line bundle on $G/P$ admits a unique $G$-linearisation. Then one can associate to a line bundle $L$ a unique character $\lambda$, given by restricting the $G$-action to the fiber over the base point.

\begin{lemma}
\label{ample}
    A homogeneous line bundle $L$ on $G/P$ with associated character $\lambda$ is ample if and only if 
    \[
    (\lambda,\alpha) > 0 \quad \text{for all } \alpha \in \Delta \backslash I.
    \]
\end{lemma}

\begin{proof}
If $P$ is reduced, this holds by \cite[II.8.5, II.4.4]{Jan}. In general, let us consider the finite morphism
\[
\sigma \colon G/P_{\text{red}} \longrightarrow G/P.
\]
The character associated to $\sigma^\ast L$ is the restriction of $\lambda$ to $P_{\text{red}}$. This, together with the fact that $L$ is ample if and only if $\sigma^\ast L$ is ample, allows us to conclude.
\end{proof}

\begin{lemma}
\label{chi}
    The character associated to the anticanonical bundle of $G/P$ is given by
    \[
    \chi = \sum_{\gamma \in \Phi^+\backslash \Phi_I} p^{\varphi(\gamma)} \gamma,
    \]
    where $\varphi$ is the associated numerical function to the parabolic subgroup $P$.
\end{lemma}

\begin{proof}
    See \cite[Proposition 3.1]{Lauritzen}: the key point is that in the proof there is no assumption on the characteristic.
\end{proof}

\begin{theorem}
\label{notFano}
    Let $n \geq 1$ be a fixed integer.\\
    There are a finite number of isomorphism classes of projective homogeneous varieties of dimension $n$ which are Fano.
\end{theorem}

\begin{proof}
    Such varieties are all of the form
    \[
    X= G/P
    \]
    where $G$ is semisimple, simply connected and $P$ is a parabolic subgroup.\\
    Up to replacing the adjoint quotient of $G$
    with its image into the automorphism group of $X$, we can assume that $G$ acts on $X$ with a finite kernel. In particular, the same is true for a maximal torus $T \subset G$; this implies that the stabiliser
    \[
    \Stab_T(x)
    \]
    is finite for a general point $x \in X$ and thus that
    \[
    n = \dim X \geq \dim T = \rank G.
    \]
    Summarizing, we have just proved that the rank of $G$ is bounded by the dimension of $X$, which is fixed and equal to $n$. Any such $G$ is semisimple and simply connected, thus product of simple factors; there are finitely many isomorphism classes of such groups with rank less or equal than $n$, thus there are finitely many possibilities for $G$.\\
    Next, we fix a Borel subgroup $B$: we can assume that $P$ contains $B$ and moreover, that $P$ does not contain the kernel of any isogeny with no central factor.\\
    \textbf{Step 1}: if the reduced part of $P$ is maximal, then by \Cref{kerphi}, $X$ is either isomorphic to a flag variety with reduced stabiliser, which is in particular Fano, or to the exotic variety with stabiliser $P_{\mathfrak{l}}$ in type $G_2$ (see \Cref{rank1_G} for this case): these are a finite number of non-isomorphic varieties, so we can exclude them.\\
    \textbf{Step 2}: Let us assume that $P$ is quasi-standard: then we have
    \[
    \Delta \backslash I = \{\beta_1 , \ldots,\beta_r\}, \quad \text{where } P_{\text{red}} = P_I.
    \]
    Up to re-arranging $\beta_1,\ldots,\beta_r$, we can find isogenies $\xi_2,\ldots , \xi_r$ with source $G$, uniquely determined by the conditions 
    \[
    \langle P,P^{\beta_i}\rangle = (\ker \xi_i) P^{\beta_i} \quad \text{and} \quad \ker \xi_2 \subseteq \ldots \subseteq \ker \xi_r,
    \]
    such that we can write $P$ as
    \[
    P = P^{\beta_1} \cap (\ker \xi_2 )P^{\beta_2} \cap \ldots \cap (\ker \xi_r) P^{\beta_r}.
    \]
    We can assume that $r\geq 2$, because the case of $r=1$ has been treated in Step 1. Let us consider a positive integer $m$, big enough such that it satisfies
    \begin{align}
    \label{emme}
    p^m > H \defeq \frac{\max_{\alpha \in \Delta} \sum_{\alpha \in \Supp(\gamma)} \vert (\gamma,\alpha) \vert} {\min_{\alpha \in \Delta, (\gamma,\alpha)<0} \vert (\gamma,\alpha)\vert}.
    \end{align}
By Remark \ref{minimale}, if we exclude a \emph{finite} number of cases, we can assume that there is some $i<r$ and some $m_0 \in \N$ such that
\begin{align}
\label{incidenza}
    \ker \xi_2 \subseteq \ldots \subseteq \ker \xi_i \subseteq G^{m_0} \quad \text{and} \quad G^{m_0+m} \subseteq \ker \xi_{i+1}\subseteq \ldots \subseteq \ker\xi_r.
\end{align}
\\
\textbf{Claim}: if $P$ satisfies (\ref{incidenza}), then $X$ is not Fano.\\
Let us write $\Phi^+\backslash \Phi_I$ as the disjoint union of $\Psi_{\geq}$ and $\Psi_{\leq}$, defined as:
\begin{align*}
    \Psi_{\geq} & \defeq\{ \gamma \in \Phi^+\backslash \Phi_I, \,\, \Supp(\gamma) \cap \{ \beta_1,\ldots, \beta_i\} = \emptyset \};\\
    \Psi_{\leq} & \defeq \{ \gamma \in \Phi^+\backslash \Phi_I, \,\, \exists j\leq i \text{ such that } \beta_j \in \Supp(\gamma)\}.
\end{align*}
The condition (\ref{incidenza}) implies that
\[
\varphi(\gamma) \leq m_0 \text{ for } \gamma \in \Psi_\leq, \quad \varphi(\gamma) \geq m_0+m \text{ for } \gamma \in \Psi_\geq.
\]
Let us fix some $\beta_l$ with $l\leq i$ and some $\delta \in \Psi_\geq$ such that
\[
(\delta,\beta_l) = -s \quad \text{for some } s>0,
\]
which exists thanks to Lemma \ref{delta} below.  Let us notice that $(\gamma,\beta_l) >0$ implies that $\beta_l \in \Supp(\gamma)$, hence
\[
(\gamma,\beta_l) \leq 0 \quad \text{for all } \gamma \in \Psi_\geq.
\]
On the other hand, let us write $\Psi_{\leq}$ as the disjoint union of 
\begin{align*}
    \Psi_{\leq}^- & \defeq\{ \gamma \in \Psi_{\leq}, \,\, (\gamma,\beta_l) \leq 0\};\\
    \Psi_{\leq}^+ & \defeq\{ \gamma \in \Psi_{\leq}, \,\, (\gamma,\beta_l) > 0\} \subset \{ \gamma \in \Phi^+, \,\, \beta_l \in \Supp(\gamma)\}.
\end{align*}
By Lemma \ref{ample} above, to prove that $X$ is not Fano it is enough to show that
\begin{align}
    \label{negativo}
    (\chi,\beta_l)<0,
\end{align}
where $\chi$ is the character associated to the anticanonical bundle of $X$.\\
By Lemma \ref{chi}, we get 
\begin{align*}
     (\chi,\beta_l) & = \sum_{\gamma \in \Psi_\leq} p^{\varphi(\gamma)} (\gamma,\beta_l) + \sum_{\gamma \in \Psi_\geq} p^{\varphi(\gamma)} (\gamma,\beta_l) \leq \sum_{\gamma \in \Psi_\leq^+} p^{\varphi(\gamma)}(\gamma,\beta_l) + \sum_{\gamma \in \Psi_\geq} p^{\varphi(\gamma)} (\gamma,\beta_l) \\
    & \leq p^{m_0} \sum_{\gamma \in \Psi_\leq^+ } (\gamma,\beta_l) + p^{m_0+m}  \sum_{\gamma \in \Psi_\geq} (\gamma,\beta_l) \leq p^{m_0} \left( \sum_{\gamma \in \Psi_\leq^+} (\gamma,\beta_l) - p^m s \right)
\end{align*}
Let us consider the integer
\[
N \defeq \sum_{\gamma \in \Psi_\leq^+} (\gamma,\beta_l) \leq \sum_{\gamma \in \Phi^+, \, \beta_l \in \Supp(\gamma)} \vert (\gamma,\beta_l)\vert 
\]
Then, by the assumption (\ref{emme}) we have that
\[
p^m > H \geq N/s,
\]
hence
\[
(\chi,\beta_l) \leq p^{m_0} (N-p^ms) <0
\]
and we have proved (\ref{negativo}).\\
\textbf{Step 3}: The last case to treat is the one of a group $G$ of type $G_2$ in characteristic $p=2$ and of a parabolic subgroup which is not standard. By Step 1 above, we can assume that the reduced part of $P$ is the Borel $B$. By \Cref{cor}, the variety $X$ is isomorphic to exactly one $G/P$ with $P$ belonging to the following list:
\[
B, \quad G^m P^{\alpha_1} \cap P^{\alpha_2}, \quad P^{\alpha_1} \cap G^m P^{\alpha}, \quad P_{\mathfrak{h}} \cap G^m P^{\alpha_2}, \quad P_{\mathfrak{l}}\cap G^m P^{\alpha_2}, \text{ for } m \geq 1.
\]
Clearly, the variety $G/B$ is Fano. Next, let us assume that $P$ is nonreduced and that $m \geq 2$, thus excluding a finite number of varieties: under this assumption, we claim that $X$ is not Fano. In order to make computations, let us recall that we have
\[
(\alpha_1,\alpha_1) = 2,\quad (\alpha_1,\alpha_2) =-3, \quad (\alpha_2,\alpha_2) = 6.
\]
Let us proceed by a case-by-case analysis, using the canonical bundle formula of Lemma \ref{chi}.\\
$\bullet$ If $P =  G^m P^{\alpha_1} \cap P^{\alpha_2}$, then
\[
\varphi(\alpha_1) = m \quad \text{and} \quad \varphi = 0 \text{ on } \Phi^+ \backslash \{\alpha_1\}. 
\]
Thus 
\[
(\chi,\alpha_2) = ((2^m+9)\alpha_1+6\alpha_2,\alpha_2) = 9-3\cdot 2^m <0.
\]
$\bullet$ If $P =  P^{\alpha_1} \cap G^m P^{\alpha_2}$, then
\[
\varphi(\alpha_2) = m \quad \text{and} \quad \varphi = 0 \text{ on } \Phi^+ \backslash \{\alpha_2\}. 
\]
Thus 
\[
(\chi,\alpha_1) = (10\alpha_1+(5+2^m)\alpha_2,\alpha_1) = 5-3\cdot 2^m <0.
\]
$\bullet$ If $P =  P_{\mathfrak{h}} \cap G^m P^{\alpha_2}$, then
\[
\varphi(\alpha_2) = m, \quad \varphi(2\alpha_1+\alpha_2) = 1 \quad  \text{and} \quad \varphi = 0 \text{ on } \Phi^+ \backslash \{\alpha_2,2\alpha_1+\alpha_2\}. 
\]
Thus 
\begin{align}
\label{neg}
(\chi,\alpha_1) = ((4+8)\alpha_1+(2^m+6)\alpha_2,\alpha_1) = 6-3\cdot 2^m <0.
\end{align}
$\bullet$ If $P = P_{\mathfrak{l}} \cap G^mP^{\alpha_2}$, then
\[
\varphi(\alpha_2) = m, \quad \varphi(\alpha_1)= \varphi(\alpha_1+\alpha_2) = 1 \quad  \text{and} \quad \varphi = 0 \text{ on } \Phi^+ \backslash \{\alpha_2,\alpha_1,\alpha_1+\alpha_2\}. 
\]
Then we get the exact same computation as in (\ref{neg}) and we can conclude.
\end{proof}

\begin{lemma}
\label{delta}
    Let $1\leq i < r$ and consider a partition of simple roots as follows:
    \[
    \Delta \backslash I = \{\beta_1 \ldots \beta_i\} \cup \{ \beta_{i+1}\ldots \beta_r\}.
    \]
    Then there is some $l\leq i$ and some $\delta \in \Phi^+\backslash \Phi_I$ such that
    \[
    \Supp(\delta) \cap \{ \beta_1,\ldots, \beta_i\} = \emptyset \quad \text{and} \quad (\delta,\beta_l) <0.
    \]
\end{lemma}

\begin{proof}
    Let us fix
    \[
    \nu \in \{\beta_1,\ldots,\beta_i\} \quad \text{and} \quad \mu \in \{ \beta_{i+1},\ldots,\beta_r \},
    \]
    such that the couple $(\nu, \mu)$ realises the minimum of the distance between the corresponding nodes in the Dynkin diagram. In particular, there is a connected segment (of minimal length) of nodes $J \subset \Delta$, having as extremes the nodes $\nu$ and $\mu$; either these two are adjacent, or the nodes in between them are all simple roots belonging to $I$. Let us consider the interior
    \[
    K \defeq J\backslash \{ \nu,\mu\} \subset I.
    \]
    Then, we can set
    \[
    \beta_l \defeq \nu \quad \text{and} \quad \delta \defeq \mu + \sum_{\alpha \in K} \alpha.
    \]
    Since either $\mu$ or some root of $K$ is adjacent to $\nu$, and moreover $\nu$ is not in the support of $\delta$, we can conclude that $(\delta,\beta_l)$ is strictly negative.
\end{proof}




\begin{thebibliography}{1}





\bibitem[BT]{BorelTits} {A. Borel, J. Tits}, {\emph{Compléments à l'article « Groupes réductifs »}},  Inst. Hautes Études Sci. Publ. Math. No. 41 (1972), 253–276.

\bibitem[Bou]{Bourbaki} {N. Bourbaki}, {\emph{Éléments de mathématique. Fasc. XXXIV. Groupes et algèbres de Lie. Chapitre IV: Groupes de Coxeter et systèmes de Tits. Chapitre V: Groupes engendrés par des réflexions. Chapitre VI: systèmes de racines.}}, Actualités Scientifiques et Industrielles, No. 1337 Hermann, Paris 1968


\bibitem[Bri]{Brion17} {M. Brion}, {\emph{Some structure theorems for algebraic groups. Algebraic groups: structure and actions}}, 53–126, Proc. Sympos. Pure Math., 94, Amer. Math. Soc., Providence, RI, 2017.


\bibitem[CGP]{CGP15} {B. Conrad, O. Gaber, G. Prasad}, {\emph{Pseudo-reductive groups. Second edition}}, New Mathematical Monographs, 26. Cambridge University Press, Cambridge, 2015. xxiv+665 pp.


\bibitem[Dem]{Demazure} {M. Demazure}, {\emph{Automorphismes et déformations des variétés de Borel}}, Invent. Math. 39 (1977), no. 2, 179–186.


\bibitem[HL]{HL93} {W. Haboush, N. Lauritzen}, {\emph{Varieties of unseparated flags}}, Linear algebraic groups and their representations (Los Angeles, CA, 1992), 35–57, Contemp. Math., 153, Amer. Math. Soc., Providence, RI, 1993.



\bibitem[Hum]{Humphreys} {J. E. Humphreys}, {\emph{Introduction to Lie algebras and representation theory. Second printing, revised.}}, Graduate Texts in Mathematics, 9. Springer-Verlag, New York-Berlin, 1978. xii+171

\bibitem[Jan]{Jan} {J.C. Jantzen}, {\emph{Representations of algebraic groups. Second edition}}, Mathematical Surveys and Monographs, 107. American Mathematical Society, Providence, RI, 2003. xiv+576 pp.

\bibitem[Lau]{Lauritzen} {N. Lauritzen}, {\emph{Splitting properties of complete homogeneous spaces}}, J. Algebra 162 (1993), no. 1, 178–193.


\bibitem[Mac]{Maccan} {M. Maccan}, {\emph{Projective homogeneous varieties of Picard rank one in small characteristic}}, \hyperlink{https://arxiv.org/abs/2303.12391}{https://arxiv.org/abs/2303.12391}


\bibitem[MO]{MatsumuraOort} {H. Matsumura, F. Oort}, {\emph{Representability of group functors, and automorphisms of algebraic schemes}}, Invent. Math. 4 (1967), 1–25.

\bibitem[Mil]{Milne} {J.S. Milne}, {\emph{Algebraic groups: the theory of group schemes of finite type over a field}}, Cambridge Studies in Advanced Mathematics, 170. Cambridge University Press, Cambridge, 2017. xvi+644 pp.





\bibitem[Sri]{Srinivasan} {S. Srinivasan}, {\emph{Motivic decomposition of projective pseudo-homogeneous varieties.}} Transform. Groups 22 (2017), no.4, 1125–1142.


\bibitem[Str]{Strade} {H. Strade}, {\emph{Simple Lie algebras over fields of positive characteristic I: Structure theory}}, De Gruyter Expositions in Mathematics, 38. Walter de Gruyter and Co., Berlin, 2004. viii+540 pp. 

\bibitem[Tot]{Totaro} {B. Totaro}, {\emph{The failure of Kodaira vanishing for Fano varieties, and terminal singularities that are not Cohen-Macaulay}},
J. Algebr. Geom. 28, No. 4, 751-771 (2019).

\bibitem[Wen]{Wenzel} {C. Wenzel},
{\emph{Classification of all parabolic subgroup-schemes of a reductive linear algebraic group over an algebraically closed field.}}, Trans. Amer. Math. Soc. 337 (1993), no. 1, 211–218.

\end{thebibliography}
\end{document}